\documentclass[11pt, reqno]{amsart}

\usepackage[utf8]{inputenc}
\usepackage[margin= 1in]{geometry}
\usepackage{graphicx}
\usepackage{caption}
\usepackage{subcaption}
\usepackage{mathtools}
\usepackage{parskip}
\usepackage[colorlinks=true, linkcolor=blue, citecolor=blue]{hyperref}
\usepackage[noabbrev,capitalise,nameinlink]{cleveref}
\usepackage{hyperref}
\crefname{figure}{figure}{}
\usepackage{xcolor}
\usepackage{amssymb}

\newtheorem{theorem}{Theorem}[section]
\newtheorem{proposition}[theorem]{Proposition}
\newtheorem{corollary}[theorem]{Corollary}

\newtheorem{problem}[theorem]{Problem}
\newtheorem{lemma}[theorem]{Lemma}

\theoremstyle{definition}
\newtheorem{definition}[theorem]{Definition}
\newtheorem{remark}[theorem]{Remark}
\newtheorem{example}[theorem]{Example} 

\begin{document}

\title{On some posets and lattices with the same height}

\author{Hoan La}

\email{hoan.la.work@gmail.com}

\begin{abstract}
    For a finite poset $\mathcal{P}$, its height $h(\mathcal{P})$ is the number of cover relations in its longest chain. When $\mathcal{P}$ is a lattice $\mathcal{L}$, we label its elements $x$ with $h(x_\downarrow) = h([\hat{0},x])$ and its cover relations $x \lessdot y$ with $h(y_\downarrow) - h(x_\downarrow)$. When a lattice $\mathcal{L}'$ extends $\mathcal{L}$, $h(x_\downarrow)_\mathcal{L} \leq h(x_\downarrow)_{\mathcal{L}'}$. We study lattices $\mathcal{L}$ and $\mathcal{L}'$ such that $h(x_\downarrow)_\mathcal{L} = h(x_\downarrow)_{\mathcal{L}'}$. We show that cover relations labeled $1$ in $\mathcal{L}$ are all cover relations that $\mathcal{L}$ and $\mathcal{L}'$ have in common. These cover relations induce a poset that we call the (long) skeletal poset $\mathrm{SK}(\mathcal{L})$. The Hasse diagram of $\mathrm{SK}(\mathcal{L})$ is the largest spanning subgraph that the Hasse diagrams of $\mathcal{L}$ and $\mathcal{L}'$ have in common. An example of lattices $\mathcal{L}$ and $\mathcal{L}'$ is the alt-Tamari lattices introduced by Chenevière, where every alt-Tamari lattice $\mathrm{alt}\text{-}\mathrm{Tam}_n$ extends the Tamari lattice $\mathrm{Tam}_n$/refines the Dyck lattice $\mathrm{Dyck}_n$ such that $h(x_\downarrow)_{\mathrm{Tam}_n} = h(x_\downarrow)_{\mathrm{alt}\text{-}\mathrm{Tam}_n}$.

    We study $\mathrm{SK}(\mathrm{Tam}_n)$ with another poset we introduce. We enumerate intervals in these posets. For a well-chosen distributive lattice, we introduce its altitude lattices, which generalize the alt-Tamari lattices $\mathrm{alt}\text{-}\mathrm{Tam}_n$ even in the case of the Dyck lattice. Altitude lattices within a family have the same number of linear intervals. They are related to each other via extensions, refinements, and embeddings of some skeletal posets. For a poset $\mathcal{P}$ with $\hat{0}$, we define its Kneser graphs $KG(k) := (V(k),E)$, where $V(k) := \{x: h(x_\downarrow) = k, 1 \leq k \leq h(\mathcal{P})\}$ and $E := \{(x,y): x_\downarrow \cap y_\downarrow =\hat{0}\}$. We give some observations about them in a reconstruction setting.
\end{abstract}

\maketitle

\section{Introduction}
\label{sec:intro} 

We recall some poset terminology as well as definitions with examples. The familiar reader may skip to \Cref{subsec:height}. A \emph{poset} $\mathcal{P} := (S, \leq)$ is a binary relation $\leq$ on the set $S$ such that for any $x,y,z \in S$: $x \leq x$; $x \leq y, y \leq x$ implies $x=y$; and $x \leq y, y \leq z$ implies $x \leq z$. Elements $x$ and $y$ are \emph{comparable} if and only if $x \leq y$ or $y \leq x$. They are \emph{incomparable} otherwise. If all elements of $\mathcal{P}$ are pairwise comparable, then $\mathcal{P}$ is a \emph{totally ordered set} or a \emph{chain} $\mathcal{C}$.

Consider $\mathbb{N} := \{1, 2, \dots \}$. For any $x,y \in \mathbb{N}$, either $x \leq y$ or $y \leq x$. $(\mathbb{N}, \leq)$ is an example of a totally ordered set. $\mathbb{N}$ can be equipped with other orderings such as divisibility, which we denote as $(\mathbb{N}, \mid)$. In $(\mathbb{N}, \mid)$, $3 \leq 6$ as $3 \mid 6$ but $4 \not\leq 6$ as $4 \nmid 6$. $(\mathbb{N}, \mid)$ is an example of a poset.

\begin{figure}[htp]
    \centering
    \includegraphics[width=0.25\linewidth]{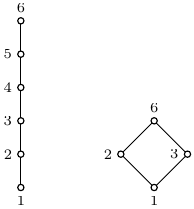}
    \caption{The interval $[1,6]$ in the posets $(\mathbb{N}, \leq)$ (left) and $(\mathbb{N}, \mid)$ (right).}
    \label{fig:N}
\end{figure}  

An interval $\mathcal{I}$ in $\mathcal{P}$ is denoted $[x,y] := \{z \in \mathcal{P}: x \leq z \leq y\}$. A cover relation $x \lessdot y$ is the interval $[x,y] := \{x,y\}$. The \emph{Hasse diagram} of a poset $\mathcal{P}$ is a directed acyclic graph $(S,\lessdot)$ with edges $x \lessdot y$, and $\mathcal{P}$ is the \emph{reflexive transitive closure} of its cover relations $\lessdot$.

$\mathcal{P} := (S, \leq_\mathcal{P})$ is a \emph{refinement} of $\mathcal{Q} := (S, \leq_\mathcal{Q})$ or $\mathcal{P}$ \emph{refines} $\mathcal{Q}$ if and only if $x \leq_{\mathcal{P}} y$ implies $x \leq_{Q} y$. Equivalently, $\mathcal{Q}$ is an \emph{extension} of $\mathcal{P}$ or $\mathcal{Q}$ \emph{extends} $\mathcal{P}$. Using our examples $(\mathbb{N}, \leq)$ and $(\mathbb{N}, \mid)$, $(\mathbb{N}, \leq)$ extends $(\mathbb{N}, \mid)$ as $x \mid y$ implies $x \leq y$.

A \emph{lattice} $\mathcal{L} := (\mathcal{P}, \wedge, \vee)$ is a poset $\mathcal{P}$ such that every pair of elements $x$ and $y$ has a \emph{meet} $x \wedge y$ and a \emph{join} $x \vee y$. To define $\wedge$ and $\vee$, we need the \emph{down set} $x_{\downarrow} := \{z \in S, z\leq x\}$ of an element $x$ and its \emph{upset} $x_\uparrow := \{z \in S, z \geq x\}$. Then, $x \wedge y$ is the unique maximal element of $x_\downarrow \cap y_\downarrow$, and $x \vee y$ is the unique minimal element of $x_\uparrow \cap y_\uparrow$. The existence of $\wedge$ and $\vee$ implies that finite lattices have a unique minimal element $\hat{0}$ and a unique maximal element $\hat{1}$, i.e., $\mathcal{L} = [\hat{0}, \hat{1}] = \hat{1}_\downarrow$. When $\mathcal{P}$ has $\wedge$ or $\vee$, then $\mathcal{P}$ is a \emph{semilattice}.

The posets $(\mathbb{N}, \leq)$ and $(\mathbb{N}, \mid)$ are examples of lattices. Explicitly, for $(\mathbb{N}, \leq)$, $x \wedge y := min(x,y)$ and $x \vee y := max(x,y)$. For $(\mathbb{N}, \mid)$, $x \wedge y := gcd(x,y)$ and $x \vee y := lcm(x,y)$. Besides $(\mathbb{N}, \leq)$ and $(\mathbb{N}, \mid)$, posets in this paper are finite.

\subsection{The height of a lattice}
\label{subsec:height}

The goal of this paper is to study finite lattices $\mathcal{L}$ via a \emph{height} (or length) function $h: \mathbb{\mathcal{P}} \rightarrow \mathbb{Z}_{\geq 0}$. We define $h$ and list some simple properties below.

\begin{proposition}[Properties of the height of a poset]
\label{prop:height}
    Let $h$ be a function $h: \mathbb{\mathcal{P}} \rightarrow \mathbb{Z}_{\geq 0}$ from a poset $\mathcal{P}$ to the nonnegative integers $\mathbb{Z}_{\geq 0}$ such that $h(\mathcal{P})$ is the number of cover relations in a longest chain in $\mathcal{P}$. The following hold for $x,y \in \mathcal{P}$:
    \begin{itemize}
        \item[(i)] if $h(x_\downarrow) = h(y_\downarrow)$, then $x=y$ or $x$ and $y$ are incomparable,
        \item[(ii)] $h([x,y]) \leq  h(y_\downarrow) - h(x_\downarrow)$,
        \item[(iii)] if a poset $\mathcal{Q}$ extends $\mathcal{P}$, then $h(x_\downarrow)_\mathcal{P} \leq h(x_\downarrow)_\mathcal{Q}$.
    \end{itemize}
\end{proposition}

\begin{proof}
    (i) Let $h(x_\downarrow) = h(y_\downarrow)$. Then, either $x=y$ or $x \neq y$. If $x \neq y$, then $x \not\in y_\downarrow$ and $y \not\in x_\downarrow$. So, $x$ and $y$ are incomparable. (i) holds as claimed.

    (ii) Consider an interval $[x,y]$ and $\mathcal{C}$ a longest chain in $y_\downarrow$. If $x \in \mathcal{C}$, then $[x,y]$ is an interval in $\mathcal{C}$, and $h(x_\downarrow) + h([x,y]) = h(y_\downarrow)$. If $x \not\in \mathcal{C}$, $h(x_\downarrow) + h([x,y]) < h(y_\downarrow)$. It follows that $h(x_\downarrow) + h([x,y]) \leq h(y_\downarrow) $. Rearranging the inequality gives (ii). 

    (iii) Since $\mathcal{Q}$ has more intervals than $\mathcal{P}$, $h(\mathcal{P}) \leq h(\mathcal{Q})$ and (iii) follows.
\end{proof}

Our motivation for using the height function $h$ to study lattices comes from the family of altitude-Tamari or alt-Tamari lattices $\mathrm{alt}\text{-}\mathrm{Tam}_n$ introduced by Chenevière~\cite{C1}. For this paper, it is not necessary to know what all of these lattices are. It suffices to know that they extend the Tamari lattice $\mathrm{Tam}_n$~\cite{C1}, a lattice introduced by Tamari~\cite{Tam}, and that $h(x_\downarrow)_{\mathrm{Tam}_n} = h(x_\downarrow)_{\mathrm{alt}\text{-}\mathrm{Tam}_n}$ (e.g., see \Cref{fig:sk_tam_3}). We show that $h(x_\downarrow)_{\mathrm{Tam}_n} = h(x_\downarrow)_{\mathrm{alt}\text{-}\mathrm{Tam}_n}$ in \Cref{sec:tam}.

For lattices $\mathcal{L}$ and $\mathcal{L}'$ such that $h(x_\downarrow)_\mathcal{L} = h(x_\downarrow)_{\mathcal{L}'}$, we observe that they have cover relations in common. Using $h$, we characterize what these cover relations are.

\begin{lemma}[Characterizing cover relations in lattices]
\label{lem:cov}
    Let $x \lessdot y$ be a cover relation in a lattice $\mathcal{L}$. We label $x \lessdot y$ with $h(y_\downarrow) - h(x_\downarrow)$. The following hold for cover relations $x \lessdot y$:
    \begin{itemize}
        \item if $\mathcal{L}$ has at least three elements, then $1 \leq h(y_\downarrow) - h(x_\downarrow) \leq h(\mathcal{\mathcal{L}}) - 1$,
        \item $x \lessdot y$ is in a longest chain in $y_\downarrow$ if and only if $h([x,y]) = h(y_\downarrow) - h(x_\downarrow) = 1$,
        \item let a lattice $\mathcal{L}'$ extends $\mathcal{L}$, then $h(x_\downarrow)_\mathcal{L} = h(x_\downarrow)_{\mathcal{L}'}$ if and only if the cover relations $x \lessdot y$ labeled $1$ in $\mathcal{L}$ are also cover relations in $\mathcal{L}'$, i.e., these cover relations are those that $\mathcal{L}$ and $\mathcal{L}'$ have in common. 
    \end{itemize}
\end{lemma} 

\begin{figure}[htp]
    \centering
    \includegraphics[width=0.8\linewidth]{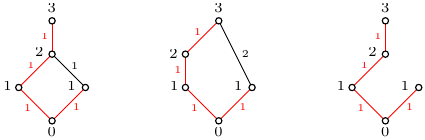}
    \caption{Examples of alt-Tamari lattices with the poset induced by \textcolor{red}{cover relations} in \textcolor{red}{red}. Elements are labeled with $h(x_\downarrow)$ and cover relations $x \lessdot y$ labeled with $h(y_\downarrow) - h(x_\downarrow)$. The leftmost lattice extends the one to its right. The rightmost poset is the induced poset.}
    \label{fig:sk_tam_3}
\end{figure}

\begin{proof}
    (i) Let $\hat{0} \lessdot z \lessdot \hat{1}$ be a chain in $\mathcal{L}$. By definition of $h$, $h(\hat{0} \lessdot z) = 1$. Since $h(\mathcal{L}) = h(\hat{1}_\downarrow)$ and $h(\hat{0} \lessdot z) = h(z_\downarrow)$, it follows that $h(z \lessdot \hat{1}) = h(\mathcal{L}) - h(\hat{0} \lessdot z) = h(\hat{1}_\downarrow) - h(z_\downarrow) = h(\mathcal{L}) - 1$. So, $1 \leq h(y_\downarrow) - h(x_\downarrow) \leq h(\mathcal{L}) - 1$.

    (ii) Suppose $x \lessdot y$ is in a longest chain $\mathcal{C}$ in $y_\downarrow$. Recall the inequality $h([x,y]) \leq h(y_\downarrow) - h(x_\downarrow)$ from (ii) of \Cref{prop:height} with $h([x,y]) = h(y_\downarrow) - h(x_\downarrow)$ if and only if $x \in \mathcal{C}$ for $\mathcal{C}$ a longest chain of $y_\downarrow$. We have $h(x \lessdot y) = h([x,y]) = h(y_\downarrow) - h(x_\downarrow) = 1$. Suppose $h([x,y]) = h(y_\downarrow) - h(x_\downarrow) = 1$. By definition of $h$, $[x,y] = x \lessdot y$. Using the argument in (ii) of \Cref{prop:height}, we have the claim as desired. 

    (iii) Suppose $\mathcal{L}'$ extends $\mathcal{L}$ such that $h(x_\downarrow)_\mathcal{L} = h(x_\downarrow)_{\mathcal{L}'}$, then there exists at least one longest chain $\mathcal{C}$ in $x_\downarrow$ such that $\mathcal{C}$ is in both $\mathcal{L}$ and $\mathcal{L}'$. By the previously proven (ii) above, all cover relations in the chains $\mathcal{C}$ are labeled $1$. Since the chains $\mathcal{C}$ are in both $\mathcal{L}$ and $\mathcal{L}'$, and $\mathcal{L}'$ extends $\mathcal{L}$, the cover relations labeled $1$ in $\mathcal{L}$ are cover relations in $\mathcal{L}$ and $\mathcal{L}'$. The backward direction, assuming $\mathcal{L}'$ extends $\mathcal{L}$, uses the same argument (by reading it backwards). We have the claims as desired.
\end{proof}

\begin{corollary}[Properties of the induced poset from cover relations labeled $1$]
\label{cor:sk}
    The cover relations labeled $1$ induce a meet-semilattice $\mathcal{P}$. The maximal chains in a down set of $\mathcal{P}$ are also the longest chains. Its Hasse diagram is the largest spanning (directed) subgraph that the Hasse diagrams of lattices $\mathcal{L}$ and $\mathcal{L}'$ have in common.
\end{corollary}

\begin{proof}
    It suffices to show that the induced poset $\mathcal{P}$ is a meet-semilattice. The rest of the corollary follows from (iii). Since each chain is in an interval $x_\downarrow = [\hat{0}, x]$ for any $x$, the induced poset $\mathcal{P}$ has a $\hat{0}$. For incomparable elements $x$ and $y$, $x_\downarrow \cap y_\downarrow$ has a unique maximal element $x \wedge y$. So, $\mathcal{P}$ has $\wedge$ and is a meet-semilattice. 
\end{proof}

We call the poset induced by cover relations labeled $1$ the \emph{skeletal poset} $\mathrm{SK}(\mathcal{L})$ of a lattice, and we say that lattices $\mathcal{L}$ and $\mathcal{L}'$ are \emph{anatomic lattices}.

\begin{definition}[Skeletal poset of a lattice]
\label{def:sk}
    Let $\mathcal{L}$ be a lattice, and $h$ the height function of a poset. The \emph{(long) skeletal poset} $\mathrm{SK}(\mathcal{L})$ is the meet-semilattice induced by cover relations $x \lessdot y$ in $\mathcal{L}$ labeled $h(y_\downarrow) - h(x_\downarrow) = 1$.
\end{definition}

\begin{definition}[Anatomic lattices]
\label{def:anatomic}
    Let $\mathcal{L}$ and $\mathcal{L}'$ be lattices such that $\mathcal{L}'$ extends $\mathcal{L}$ and $h(x_\downarrow)_\mathcal{L} = h(x_\downarrow)_{\mathcal{L}'}$. We call $\mathcal{L}$ and $\mathcal{L}'$ \emph{anatomic lattices}. Whenever lattices $\mathcal{L}'$ extend the same lattice $\mathcal{L}$, or lattices $\mathcal{L}$ refine the same lattice $\mathcal{L}'$, we say that the lattices $\mathcal{L}$ and $\mathcal{L}'$ are \emph{anatomically related}. 
\end{definition}

\subsection{Organization of the paper}
\label{subsec:org}

As mentioned, we elaborate on alt-Tamari lattices introduced by Chenevière. Specifically, we show that they are anatomically related. Since every alt-Tamari lattice refines the Dyck lattice $\mathrm{Dyck}_n$ and extends the Tamari lattice $\mathrm{Tam}_n$~\cite{C1}, it suffices to show that the Dyck and Tamari lattices are anatomic lattices, which we prove in \Cref{sec:tam}. We use definitions of the Dyck and Tamari lattices on Dyck paths from~\cite{BB} to do so.

From \Cref{sec:sk_tam} onward, the paper splits into three independent sections. \Cref{sec:sk_tam} and \Cref{sec:alt} focus on introducing new posets and studying them. \Cref{sec:kneser} is a short note where we suggest a setting to use Kneser graphs that, to our knowledge, has not been considered in the literature; we explain this note shortly. In each of these sections, we discuss further directions for research.

In \Cref{sec:sk_tam}, we study the skeletal poset $\mathrm{SK}(\mathrm{Tam}_n)$, which we call the Dyck-Tamari poset $\mathrm{DTam}_n$. Cover relations in $\mathrm{DTam}_n$ are equinumerous with another poset we introduce. This poset has the property that all of its cover relations are those that the Noncrossing partition lattice, or Kreweras lattice, and the Tamari lattice have in common. This observation led us to define the \emph{Kreweras-Tamari poset} $\mathrm{KTam}_n$. Since $\mathrm{KTam}_n$ satisfies properties similar to $\mathrm{DTam}_n$, we call it the wide skeletal poset of the Tamari lattice. In turn, we call $\mathrm{DTam}_n$ the long skeletal poset of the Tamari lattice. We enumerate the intervals in the skeletal posets of the Tamari lattice. Intervals in $\mathrm{DTam}_n$ are enumerated by the number of ternary trees, which also counts intervals in the Kreweras lattice~\cite{E}. Intervals in $\mathrm{KTam}_n$ are enumerated by the number of hex trees, a subset of the ternary trees. We end \Cref{sec:sk_tam} with definitions of the long and wide skeletal posets of $m$-Tamari lattices and suggest further generalizations.   

\begin{figure}[htp]
    \centering
    \includegraphics[width=0.5\linewidth]{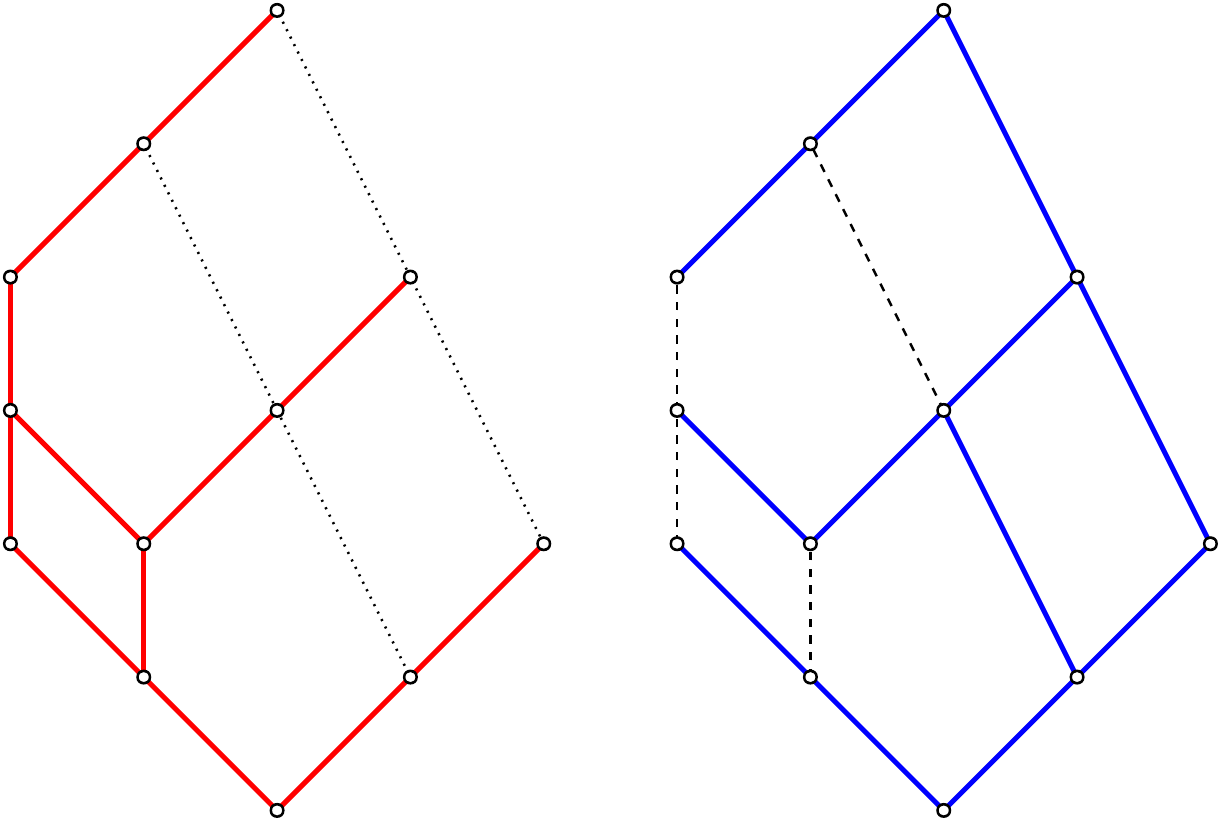}
    \caption{The Hasse diagrams of the \textcolor{red}{long skeletal poset} (left) and the \textcolor{blue}{wide skeletal poset} (right) of the $m$-Tamari lattice (for $m=2$) in \textcolor{red}{red} and \textcolor{blue}{blue} respectively. $\cdots$ and $- -$ are cover relations in the $m$-Tamari lattice.}
    \label{fig:sk_2_tam}
\end{figure}

In \Cref{sec:alt}, we introduce a generalization of the alt-Tamari lattices. Besides being anatomically related, alt-Tamari lattices have the same number of linear intervals, i.e., intervals that are chains. We observe that this property holds for other families of lattices, which we call \emph{altitude lattices}. Altitude lattices within the same family refine a well-chosen distributive lattice. In the case of the alt-Tamari lattices, the well-chosen distributive lattice is the Dyck lattice; even in this case, to the best of our knowledge, some altitude lattices are new. To introduce altitude lattices, we first make some simple order-theoretic observations about refining intervals of certain posets. Then, we consider embeddings of distributive lattices into the grid. These embeddings show us how to refine a distributive lattice locally, i.e., refining an interval, to get polygonal lattices that are embeddable into the grid, and have the same number of linear intervals. We recommend that the reader see \Cref{exm:alt} or the following figures in \Cref{sec:alt}: \Cref{fig:samples}, \Cref{fig:dyck_to_tam} and \Cref{fig:camb}.

\begin{example}[A family of altitude lattices]
\label{exm:alt}
    The figure below is an example of altitude lattices within the same family that are embedded into the grid (see the dotted lines), followed by details about their linear intervals.
    \begin{figure}[htp]
        \centering
        \includegraphics[width=0.8\linewidth]{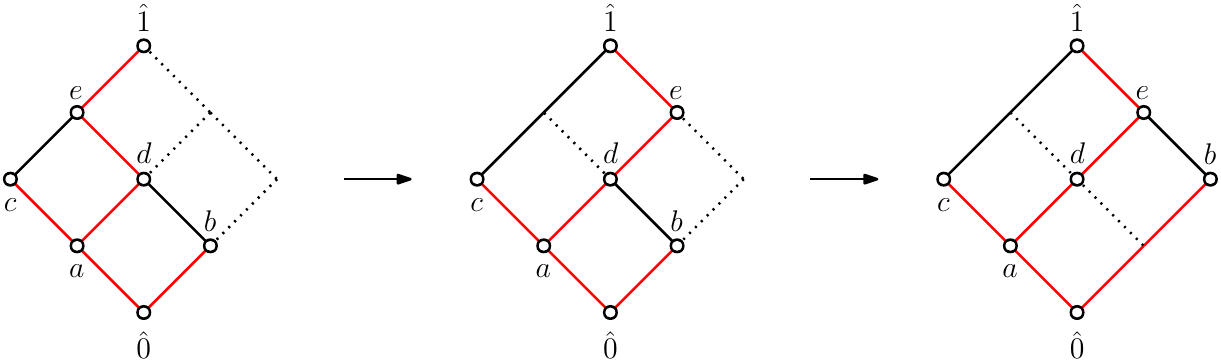}
        \caption{A family of altitude lattices. The lattice on the left side of an arrow is refine by the lattice on the right side of an arrow. The \textcolor{red}{cover relations} in \textcolor{red}{red} induce the skeletal poset of the rightmost lattice.}
        \label{fig:alt}
    \end{figure}
    \begin{itemize}
        \item There are $7$ elements or one-element chains.
        \item There are $8$ cover relations or two-element chains.
        \item There are $5$ intervals that are three-element or four-element chains, they are
        \begin{itemize}
            \item for the leftmost lattice: $[\hat{0},c], [b,e], [b,\hat{1}], [c, \hat{1}], [d, \hat{1}]$,
            \item for the middle lattice: $[\hat{0},c], [a,e], [b,e], [b,\hat{1}], [d,\hat{1}]$,
            \item for the rightmost lattice: $[\hat{0},c], [\hat{0},d], [a,e], [b, \hat{1}],[d, \hat{1}]$.
        \end{itemize}
    \end{itemize}
\end{example}

In \Cref{sec:kneser}, we use the height function $h$ to define \emph{Kneser graphs}, first studied by Kneser~\cite{K} then famously by Lovász~\cite{Lov} (notably for his use of the  Borsuk-Ulam theorem). A Kneser graph has as vertices the $k$-element subsets of $[n] := \{1,2, \dots, n\}$, and two $k$-element subsets $x$ and $y$ are adjacent when their intersection is empty. Observe that the subsets of $[n]$ ordered by inclusion define the Boolean lattice, and the empty set is its minimal element. So, a Kneser graph has as vertices the elements $x$ where $h(x_\downarrow)=k$, and vertices $x$ and $y$ are adjacent when $x_\downarrow \cap y_\downarrow = \varnothing = \hat{0}$. This definition in the poset setting using $h$ replaces the Boolean lattice with any poset with $\hat{0}$. Then, by considering all elements $x \neq \hat{0}$, we obtain the \emph{zero-divisor graph}, first introduced by Beck~\cite{B} and then generalized to the poset setting by Halaš and Jukl~\cite{HJ}.

By definition, Kneser graphs are subgraphs of the zero-divisor graph. Similarly, the zero-divisor graph is a subgraph of the \emph{incomparability graph} where vertices $x$ and $y$ are adjacent if and only if $x$ and $y$ are incomparable. The incomparability graph, dually, the comparability graph, determine, up to duality ($x \leq y$ to $x \geq y$), almost all posets~\cite{KR, Möh}. For posets that are uniquely determined by their zero-divisor graphs, these graphs can be thought of as sparser incomparability graphs. Posets with the same zero-divisor graph may not have the same Kneser subgraphs (and vice versa, which we will explain then). For example, the number of elements $x$ such that $h(x_\downarrow) = k$ may differ between these posets. With this observation, we consider how the zero-divisor graph, more generally the incomparability graph, and its Kneser subgraphs interact in a reconstruction setting.

As a final remark regarding the organization of this paper, this is a preliminary version of our work. The purpose of this paper is to give some initial details to our ideas, which, to the best of our knowledge, have not been explored in the literature. We will expand upon \Cref{sec:sk_tam} and \Cref{sec:alt} in~\cite{L1, L2}.

\section{The Dyck and Tamari lattices}
\label{sec:tam}

The Dyck and Tamari lattices are defined on sets whose number of elements is a Catalan number. The Catalan numbers $\mathrm{Cat}_n$ have a rich history (see~\cite{S}) and a simple recursion; explicitly, $\mathrm{Cat}_0 = 1, \mathrm{Cat}_n = \sum_{i=0}^{n-1}\mathrm{Cat}_i\mathrm{Cat}_{n-i-1}$ or $\mathrm{Cat}(x) = 1 + x\mathrm{Cat}(x)^2$. It became natural for the Dyck and Tamari lattices to be some of the most well-studied lattices in Combinatorics. For the familiar reader, the Dyck lattice is better known as the Stanley lattice or the principal order ideal of the staircase shape in the Young lattice. We recommend the surveys~\cite{MPS, Pons}.

\subsection{Bracket vectors and Dyck paths}
\label{subsec:bracket_dyck}

For visual clarity (e.g., see \Cref{fig:cat_bij}), we use $\emph{Dyck paths}$ to define the Dyck and Tamari lattices and then $\emph{bracket vectors}$ (first introduced by Huang and Tamari~\cite{HT}) to study the height of these lattices. These objects are defined below.

\begin{definition}[Bracket vectors, Definition $9.1$ in~\cite{BW}]
\label{def:bracket}
    A vector $v := (v_1,v_2,...,v_n)$ is a \emph{bracket vector} of size $n \in \mathbb{Z}_{\geq 0}$ if $v_i \in \mathbb{N}$ for $1 \leq i \leq n$ such that
    \begin{itemize}
        \item[(i)] $0 \leq v_i \leq n-i$, and
        \item[(ii)] $v_{i+j} \leq v_i - j$, for $0 \leq j \leq v_i$.
    \end{itemize}
\end{definition}

\begin{definition}[Dyck paths, folklore]
\label{def:dyck_path}
    A word $w := w_1w_2 \cdots w_{2n}$ is a \emph{Dyck path} of size $n \in \mathbb{Z}_{\geq 0}$ if the letters $w_i \in \{u,d\}$, for $1 \leq i \leq 2n$, such that 
    \begin{itemize}
        \item[(i)] the number of letters $w_i = u$ is equal to the number of letters $w_i = d$, and
        \item[(ii)] at any index $i$, the number of letters $w_j = u$ is at least the number of $w_j = d$ for $j \leq i$. 
    \end{itemize} 
\end{definition}

\begin{example}[Bracket vectors and Dyck paths for $n=3$ and nonexamples]
    Bracket vectors for $n=3$ are $(0,0,0), (1,0,0), (0,1,0), (2,0,0), (2,1,0)$, and the Dyck paths for $n=3$ are $u_1d_2u_3d_4u_5d_6$, $u_1u_2d_3d_4u_5d_6$, $u_1d_2u_3u_4d_5d_6$, $u_1u_2d_3u_4d_5d_6$, and $u_1u_2u_3d_4d_5d_6$. $(1,0,1)$ is not a bracket vector since $v_3 = 1 > 3-3=0$ and does not satisfy (i) of \Cref{def:bracket}. $(1,1,0)$ is not a bracket vector since $v_2 = 1 > v_1 - 1 = 0$ and does not satisfy (ii) of \Cref{def:bracket}. $u_1u_2u_3u_4d_5d_6$ is not a Dyck path since there are $4$ $u$ and $2$ $d$ and it does not satisfy (i) of \Cref{def:dyck_path}. $u_1d_2d_3u_4u_5d_6$ is not a Dyck path since there are $1$ $u$ and $2$ $d$ at the index $3$ which does not satisfy (ii) of \Cref{def:dyck_path}.
\end{example}

It remains to show that bracket vectors and Dyck paths are indeed counted by Catalan numbers. We give a bjiection between these sets afterwards.

\begin{proposition}[Counting bracket vectors and Dyck paths]
\label{prop:cat}
    The sets of bracket vectors and Dyck paths are each counted by Catalan numbers.
\end{proposition}

\begin{proof}
    For $v$ a bracket vector, fix $v_1 = k$ for $0 \leq k \leq n-1$ by \Cref{def:bracket}. Let $v'$ and $v''$ be such that $v' = (v_2, v_3, \dots, v_{k+1})$ and $v'' = (v_{k+2}, v_{k+2}, \dots, v_{n})$. Note that the empty vector is vacuously a bracket vector. Since $v$ is a bracket vector, $v'$ and $v''$ are bracket vectors by \Cref{def:bracket}. We have $\mathrm{Cat}(x) = 1 + \mathrm{Cat}(x)^2$ as desired.

    For $w$, a Dyck path, fix the Dyck path with minimal size that starts with $w_1 = u$. Such a path ends with $w_{2k} = d$ for $1 \leq k \leq n$ by \Cref{def:dyck_path}. Let $w'$ and $w''$ be words such that $w' = w_2w_3 \cdots w_{2k-1}$ and $w'' = w_{2k+1}w_{2k+2} \cdots w_{2n}$. Note that the empty word is vacuously a Dyck path. Since $w$ is a Dyck path, $w'$ and $w''$ are Dyck paths by \Cref{def:dyck_path}. We have $\mathrm{Cat}(x) = 1 + \mathrm{Cat}(x)^2$ as desired.
\end{proof}

\begin{corollary}[A bijection between bracket vectors and Dyck paths]
\label{cor:bij}
    There exists a bijection between bracket vectors and Dyck paths.
\end{corollary}

\begin{figure}[htp]
    \centering
    \includegraphics[width=0.8\linewidth]{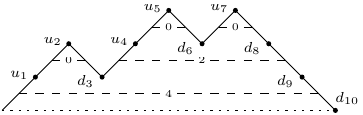}
    \caption{A summary of the bijection between bracket vectors and Dyck paths below. For each paired $u$ and $d$, there is an integer counting pairs of letters $u$ and $d$ between such a pair. The bracket vector here is $(4,0,2,0,0)$ and the Dyck path here is $u_1u_2d_3u_4u_5d_6u_7d_8d_9d_{10}$.}
    \label{fig:cat_bij}
\end{figure}

\begin{proof}
    For $v$ a bracket vector, we use $v_i$ for $1 \leq i \leq n$ to write a Dyck path $w$ as follows:
    \begin{itemize}
        \item let $w := w_1w_2 \cdots w_{2n}$ be a word, and $I := \{1,2, \dots ,2n\}$ be the indexing set,
        \item let $w_1=u$ and $w_{1+2v_1+1} =d$, then remove the indices $1$ and $1+2v_1+1$ from $I$,
        \item for $v_{i \geq 2}$, let $w_j = u$ be with index $j$ such that $j := min(I)$ and $w_{j+2v_i+1} = d$, then \newline remove indices $j$ and $j+2v_i+1$ from $I$,
        \item repeat until $I$ is empty.
    \end{itemize}
    By construction, there are $n$ many $w_k = u$ and $n$ many $w_k = d$ for $1 \leq k \leq 2n$. The word $w$ satisfies (i) of \Cref{def:dyck_path}. For any $w_j=u$ and $w_{j+2v_i+1} = d$ as described above (including $j=1$), the index of $d$ is always greater than the index of $u$. It follows that at any index of $w$, the number of $w_k = u$ is at least the number of $w_k = d$. So, $w$ satisfies (ii) of \Cref{def:dyck_path}, and $w$ is a Dyck path.

    For $w$, a Dyck path, observe that for every $w_{k'} = u$ for $1 \leq k' \leq 2n-1$, there exists $w_{k'+2k''+1} = d$ for $0 \leq k'' \leq n-1$ such that $w_{k'+1}w_{k'+2} \cdots w_{k'+2k''}$ is a Dyck path, which may be the empty path. Observe that $0 \leq k'' \leq n-1$ is (i) of \Cref{def:bracket}. We write $v_i = k''$ for $1 \leq i \leq n$, where $v := (v_1, v_2, \dots, v_n)$. The indices $i$ represent the $i$-th paired letters $w_{k'} = u$ and $w_{k'+2k''+1} = d$. So, $v_i$ is read as the value $k''$ of the $i$-th paired letters. For the $(i+j)$-th paired letters such that $0 \leq j \leq v_i$, the indices of the $(i+j)$-th paired letters are between $k'$ and $k'+2k''+1$. Thus, $v_{i+j} \leq v_i-j$, and $v$ satisfies (ii) of \Cref{def:bracket}, and $v$ is a bracket vector. We have a bijection between bracket vectors and Dyck paths as claimed.
\end{proof}

\subsection{The Dyck and Tamari lattices are anatomic lattices.}
\label{subsec:d_and_t_anatomic}

We turn our attention to the Dyck and Tamari lattices. We define these lattices as the reflexive transitive closures of cover relations defined on Dyck paths. By the end of this section, we show that they are anatomic lattices.

\begin{definition}[Dyck lattice or Stanley lattice in~\cite{BB}]
\label{def:dyck}
    The \emph{Dyck lattice} $\mathrm{Dyck}_n$ is the reflexive transitive closure of cover relations $w = AduB \lessdot w' = AudB$, where $w$ and $w'$ are Dyck paths of size $n$, and $A$ and $B$ are nonempty. 
\end{definition}

\begin{figure}[htp]
    \centering
    \includegraphics[width=0.5\linewidth]{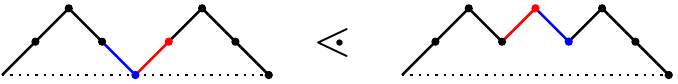}
    \caption{$u_1u_2d_3\textcolor{blue}{d_4}\textcolor{red}{u_5}u_6d_7d_8 \lessdot u_1u_2d_3\textcolor{red}{u_4}\textcolor{blue}{d_5}u_6d_7d_8$ in $\mathrm{Dyck}_4$.}
    \label{fig:cover_dyck}
\end{figure}

\begin{definition}[Tamari lattice~\cite{BB}]
\label{def:tam}
    The \emph{Tamari lattice} $\mathrm{Tam}_n$ is the reflexive transitive closure of cover relations $w = AdDB \lessdot w' = ADdB$, where $w$ and $w'$ are Dyck paths of size $n$, $A$ is nonempty, and $D = u \cdots d$ is the Dyck path of minimal size.
\end{definition}

\begin{figure}[htp]
    \centering
    \includegraphics[width=0.5\linewidth]{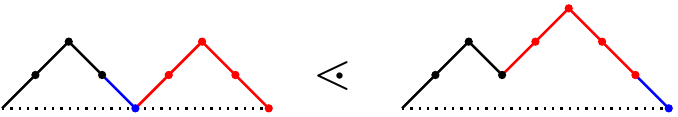}
    \caption{$u_1u_2d_3\textcolor{blue}{d_4}\textcolor{red}{u_5u_6d_7d_8} \lessdot u_1u_2d_3\textcolor{red}{u_4u_5d_6d_7}\textcolor{blue}{d_8}$ in $\mathrm{Tam}_4$.}
    \label{fig:cover_tam}
\end{figure}

By \Cref{def:dyck} and \Cref{def:tam} with \Cref{fig:cover_dyck} and \Cref{fig:cover_tam} as examples, it can be seen that the Dyck lattice is an extension of the Tamari lattice. We give a short proof of this claim below. 

\begin{proposition}[Dyck lattice extends Tamari lattice]
\label{prop:d_extends_t}
    The Dyck lattice $\mathrm{Dyck}_n$ is an extension of the Tamari lattice $\mathrm{Tam}_n$.
\end{proposition}

\begin{proof}
    We need to show that $w \leq w'$ in $\mathrm{Tam}_n$ implies $w \leq w'$ in $\mathrm{Dyck}_n$, where $w$ and $w'$ are Dyck paths of size $n$. Consider $w = AduB \lessdot w' = AudB$ in $\mathrm{Dyck}_n$. The number of letters $u$ at the index of $d$ as written in $AduB$ is one less than the number of letters $u$ as written in $AudB$ at the same index. It follows that $w \leq w'$ in $\mathrm{Dyck}_n$ if and only if the number of letters $u$ in $w$ at any index $i$ for $1 \leq i \leq 2n$ is at most the number of letters $u$ in $w'$ at the same index. We say that $w \leq w'$ in $\mathrm{Dyck}_n$ if and only if $w$ is \emph{weakly below} $w'$. Consider $w = AdDB \lessdot w' = ADdB$ in $\mathrm{Tam}_n$. Observe that $w$ is weakly below $w'$. By transitivity, $w \leq w'$ in $\mathrm{Tam}_n$ implies $w \leq w'$ in $\mathrm{Dyck}_n$. We conclude that the Dyck lattice $\mathrm{Dyck}_n$ is an extenstion of the Tamari lattice $\mathrm{Tam}_n$.
\end{proof} 

To show that the Dyck lattice $\mathrm{Dyck}_n$ and the Tamari lattice $\mathrm{Tam}_n$ are anatomic lattices, we need to prove that $h(x_\downarrow)_{\mathrm{Tam}_n} = h(x_\downarrow)_{\mathrm{Dyck}_n}$. We do using Dyck paths. With bracket vectors, we compute $h(x_\downarrow)$ for any $x$ in $\mathrm{Dyck}_n$ and $\mathrm{Tam}_n$. First, we show how bracket vectors are used to determine when Dyck paths $w$ and $w'$ are comparable in the Tamari lattice. 

\begin{lemma}[Cover relations in terms of bracket vectors]
\label{lem:cover_bracket}
    For bracket vectors $v, v'$, the pair $v = (v_1, v_2, \dots, v_i, \dots, v_n), v' = (v_1, v_2, \dots, v_i' = v_i + v_{(i+v_i+1)} + 1, \dots, v_n)$ for $1 \leq i \leq n$ is a cover relation $v \lessdot v'$ in the Tamari lattice $\mathrm{Tam}_n$.
\end{lemma}

\begin{proof}
    Recall the bijection from \Cref{cor:bij} that for any bracket vector $v = (v_1, v_2, \dots, v_n)$, where $v_k$ for $1 \leq k \leq n$ is the number of pairs of letters $u$ and $d$ between the $k$-th paired letters $u$ and $d$. Observe that a $k$-th paired letters $u$ and $d$ define a Dyck path that starts with $u$ and ends with $d$. Let $v$ be the bracket vector of the Dyck path $AdDB$ in $AdDB \lessdot ADdB$ in $\mathrm{Tam}_n$. The letter $d$ of $AdDB$ is paired with a letter $u$, and we call this pair the $i$-th pair. Let $v_i$ correspond to the $i$-th pair. Since $D$ is of minimal size, the index $k$ of $v_k$ corresponding to $D$ is $i+v_i+1$. Then, the bracket vector $v'$ of $ADdB$ is $(v_1, v_2, \dots, v_i + v_{(i+v_i+1)} + 1, \dots, v_n)$. So, $v \lessdot v'$ in $\mathrm{Tam}_n$. 
\end{proof}

\begin{corollary}[Tamari lattice defined on bracket vectors, Definition $9.1$ in~\cite{BW}]
\label{cor:bracket_tam}
    For bracket vectors $v$ and $v'$ of size $n$ corresponding to Dyck paths $w$ and $w'$ of size $n$, respectively, $w \leq w'$ in $\mathrm{Tam}_n$ if and only if $v \leq v'$, where $v_i \leq v_i'$ for $1 \leq i \leq n$. Equivalently, bracket vectors $v$ and $v'$ ordered by $\leq$ where $v \leq v'$ for $v_i \leq v_i'$ define the Tamari lattice $\mathrm{Tam}_n$.
\end{corollary}

For the Dyck and Tamari lattices to be anatomic lattices, it remains to show that $h(x_\downarrow)_{\mathrm{Tam}_n} = h(x_\downarrow)_{\mathrm{Dyck}_n}$. By (iii) of \Cref{lem:cov}, it suffices to characterize the cover relations labeled $1$ in $\mathrm{Tam}_n$, which we do using Dyck paths. Furthermore, we compute $h(x_\downarrow)$ for any $x$ in the Dyck and Tamari lattices using bracket vectors.

\begin{theorem}[Height of the Dyck and Tamari lattices]
\label{thm:dtam}
    The Dyck lattice $\mathrm{Dyck}_n$ and the Tamari lattice $\mathrm{Tam}_n$ are anatomic lattices. Consequently, the alt-Tamari lattices are anatomically related. In addition, for elements $x$ and $y$ in $\mathrm{Dyck}_n$ and $\mathrm{Tam}_n$ with bracket vectors $v = (v_1,v_2, \dots, v_n)$ and $v' = (v_1',v_2', \dots, v_n')$, $h(x_\downarrow) = \sum_{i=1}^{n} v_i$.  
\end{theorem}

\begin{proof}
    We characterize the cover relations $x \lessdot y$ in the longest chains of $y_\downarrow$ of the Dyck lattice $\mathrm{Dyck}_n$. By \Cref{def:dyck}, cover relations in $\mathrm{Dyck}_n$ are of the form $AduB \lessdot AudB$. Consider the difference in the numbers of letters $u$ of $AduB$ and $AudB$. This difference is $1$ at the index of $d$ as written in $AduB$ or $u$ as written in $AudB$. Otherwise, it is $0$. Since all cover relations are of the form $AduB \lessdot AudB$, by transitivity, all maximal chains in $y_\downarrow$ have the same number of cover relations. The sum of the differences along the cover relations is the number of cover relations. 

    Consider $AduB \lessdot AudB$ such that $B = dB'$, i.e., $AdudB' \lessdot AuddB'$. Observe that $ud$ is the Dyck path of size $1$. Consider the Tamari lattice $\mathrm{Tam}_n$. By \Cref{def:tam}, its cover relations are of the form $AdDB \lessdot ADdB$, where $D$ is the Dyck path of minimal size following $d$ as written in $AdDB$. If the size of $D$ is $1$, then $D=ud$ and $AdDB = AdudB \lessdot ADdB = AuddB$. It follows that $AdudB \lessdot AuddB$ is a cover relation in $\mathrm{Dyck}_n$ and $\mathrm{Tam}_n$.

    By \Cref{prop:d_extends_t} and \Cref{prop:height}, $\mathrm{Dyck}_n$ extends $\mathrm{Tam}_n$ and $h(x_\downarrow)_{\mathrm{Tam}_n} \leq h(x_\downarrow)_{\mathrm{Dyck}_n}$. The above arguments show that cover relations $AdudB \lessdot AuddB$ are those in the longest chains in $x_\downarrow$ for any $x$ in $\mathrm{Dyck}_n$ and $\mathrm{Tam}_n$. By (iii) of \Cref{lem:cov}, $h(x_\downarrow)_{\mathrm{Tam}_n} = h(x_\downarrow)_{\mathrm{Dyck}_n}$. It follows by \Cref{def:anatomic} that $\mathrm{Dyck}_n$ and $\mathrm{Tam}_n$ are anatomic lattices, and by ~\cite{C1}, as mentioned in \Cref{subsec:org}, that the alt-Tamari lattices are anatomically related.

    Using \Cref{lem:cover_bracket}, $AdudB \lessdot AuddB$ is $v = (v_1, v_2, \dots, v_i, \dots, v_n) \lessdot v' = (v_1, v_2, \dots, v_i + v_{(i+v_i+1)} + 1, \dots, v_n)$ such that $v_{(i+v_i+1)} = 0$, i.e., $v = (v_1, v_2, \dots, v_i, \dots, v_n) \lessdot v' = (v_1, v_2, \dots, v_i + 1, \dots, v_n)$. By transitivity, $h(x_\downarrow) = \sum_{i=1}^{n} v_i$ for any $x$ in $\mathrm{Dyck}_n$ and $\mathrm{Tam}_n$. 
\end{proof}

\begin{corollary}[Height of intervals in the Dyck lattice]
\label{cor:h_intervals}
    For comparable elements $x$ and $y$ in $\mathrm{Dyck}_n$ with bracket vectors $(v_1,v_2, \dots, v_n)$ and $(v_1',v_2', \dots, v_n')$ of $x$ and $y$, respectively, $h([x,y]) = \sum_{i=1}^{n} v_i' - v_i$.
\end{corollary}

\begin{proof}
    Recall from (ii) of \Cref{lem:cov} that $h([x,y]) = h(y_\downarrow) - h(x_\downarrow)$ for $[x,y]$ in a longest chain in a lattice. Let $(v_1,v_2, \dots, v_n)$ and $(v_1',v_2', \dots, v_n')$ be bracket vectors of $x$ and $y$, respectively. Since all maximal chains in $\mathrm{Dyck}_n$ are longest chains, $h([x,y]) = h(y_\downarrow) - h(x_\downarrow) = \sum_{i=1}^{n} v_i' - v_i$ for any interval $[x,y]$ in $\mathrm{Dyck}_n$.
\end{proof}

Cover relations $AdudB \lessdot AuddB$, by (iii) of \Cref{lem:cov}, are cover relations that the Dyck and Tamari lattices have in common. By \Cref{cor:sk}, they induce a semilattice. By \Cref{def:sk}, it is the skeletal poset $\mathrm{SK}(\mathrm{Tam}_n)$, which we call the \emph{Dyck-Tamari poset} $\mathrm{DTam}_n$.

\begin{definition}[Dyck-Tamari poset]
\label{def:dtam}
    The \emph{Dyck-Tamari poset} $\mathrm{DTam}_n$ is the reflexive transitive closure of cover relations $w = AdudB \lessdot w' = AuddB$, where $w$ and $w'$ are Dyck paths of size $n$ and $A$ is nonempty.
\end{definition}

\begin{figure}[htp]
    \centering
    \includegraphics[width=0.5\linewidth]{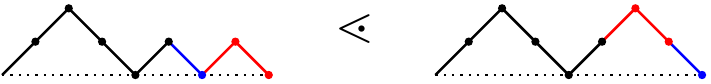}
    \caption{$u_1u_2d_3d_4u_5\textcolor{blue}{d_6}\textcolor{red}{u_7d_8} \lessdot u_1u_2d_3d_4u_5\textcolor{red}{u_6d_7}\textcolor{blue}{d_8}$ in $\mathrm{DTam}_4$.}
    \label{fig:cover_dtam}
\end{figure}

We end this section by counting cover relations in the Dyck-Tamari poset (for \Cref{sec:sk_tam}) using an idea of Galor and highlight a remark by Björner and Wachs (for \Cref{sec:alt}; also see \Cref{fig:tam}).

\begin{proposition}[Number of cover relations in the Dyck-Tamari poset]
\label{prop:num_cov}
    The number of cover relations $AdudB \lessdot AuddB$ for Dyck paths $AdudB, AuddB$ of size $n + 1$ is $n\mathrm{Cat}_{n}$, where $\mathrm{Cat}_0 = 1$ and $\mathrm{Cat}_n = \sum_{i=0}^{n-1} \mathrm{Cat}_i\mathrm{Cat}_{n-i-1}$ is a Catalan number.
\end{proposition}

\begin{proof}
    The idea for this proof is due to Galor. For any Dyck path $w$ of size $n$, there are $n$ letters $d$ by \Cref{def:dyck_path}, so $n$ ways to write $w = AdB$ or $w' = AdudB$, where $w'$ is a Dyck path of size $n+1$. Thus, the total number of ways is $nC_n$. Since $\lessdot : AdudB \rightarrow AuddB$ is a one-to-one, or injective, map, it follows that there are $nC_n$ cover relations $AdudB \lessdot AuddB$.
\end{proof}

\begin{remark}[Embedding of the Tamari lattice, 10.3. Example in~\cite{BW}]
\label{rem:embed}
Recall that bracket vectors $v$ ordered by $\leq$ define the Tamari lattice. Since $v \in \mathbb{Z}_{\geq 0}^{d-1}$ for $d$ the size of $v$, the Tamari lattice can be embedded into $(\mathbb{Z}_{\geq 0}^{d-1}, \leq)$, i.e., it is an induced poset of $x_\downarrow$ where $x = (n-1, n-2, \dots, 1) \in \mathbb{Z}_{\geq 0}^{d-1}$.  
\end{remark}

\begin{figure}[htp]
    \centering
    \includegraphics[width=0.8\linewidth]{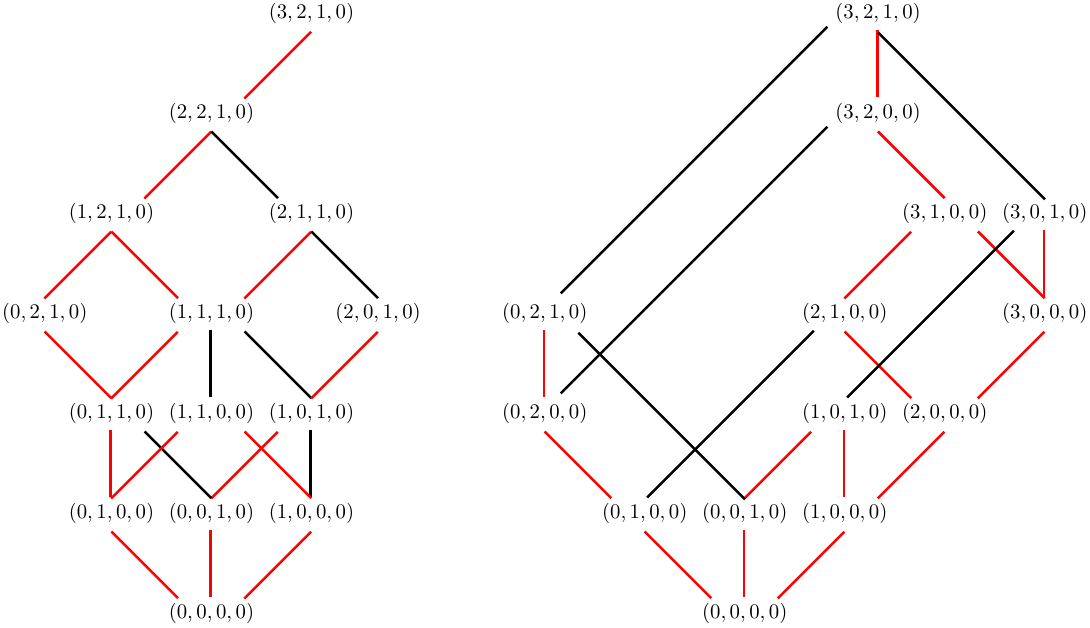}
    \caption{The Dyck lattice $\mathrm{Dyck}_4$ (left) and the Tamari lattice $\mathrm{Tam}_4$ (right) as induced posets in the grid $\mathbb{Z}_{\geq 0}^3$. Note the added $0$ is for comparison with bracket vectors. \textcolor{red}{Cover relations} in \textcolor{red}{red} induce the Hasse diagram of the Dyck-Tamari poset $\mathrm{DTam}_4$; see \Cref{sec:alt} for further details.}
    \label{fig:tam}
\end{figure}

\section{Skeletal posets of Tamari lattices}
\label{sec:sk_tam}

This section is independent of \Cref{sec:alt} and \Cref{sec:kneser}. The goal here is to study not one, but two ``skeletal'' posets of the Tamari lattice. We enumerate the intervals in these posets and map intervals in one skeletal poset to the other. The section ends with generalizations of the skeletal posets for the $m$-Tamari lattice.  

\subsection{The Kreweras-Tamari poset}
\label{subsec:w_sk}

Recall that cover relations in $\mathrm{DTam}_n$ are of the form $AdudB \lessdot AuddB$ by \Cref{def:dtam}. Consider $AuduB$. Using the same argument from \Cref{prop:num_cov}, there are $nC_n$ ways to write a Dyck path $AuduB$ of size $n+1$. Since any letter $u$ is uniquely paired with a letter $d$ as noted in the bijection from \Cref{cor:bij}, we may consider $AudDB = A'dDB$ instead, where $D$ is the Dyck path of minimal size. By \Cref{def:tam}, $A'dDB \lessdot A'DdB$ is a cover relation in the Tamari lattice $\mathrm{Tam}_n$. By the same argument from the proof of \Cref{prop:num_cov}, there are $nC_n$ cover relations $AudDB \lessdot AuDdB$, where $AudDB$ and $AuDdB$ are Dyck paths of size $n+1$. For emphasis, this observation is also due to Galor. 

Besides their enumeration, cover relations $AudDB \lessdot AuDdB$ satisfy another property similar to $AdudB \lessdot AuddB$. Consider the lattice defined below. 

\begin{definition}[Kreweras lattice~\cite{BB}]
\label{def:krew}
    The Noncrossing partition or \emph{Kreweras lattice} $\mathrm{Krew}_n$ is the reflexive transitive closure of cover relations $w = Aud^k\mathfrak{D}B \lessdot w' = Au\mathfrak{D}d^kB$, where $d^k = dd \cdots d$, i.e., consecutive $k$ letters $d$ for $1 \leq k \leq n-1$, $w$ and $w'$ are Dyck paths of size $n$, and $\mathfrak{D} = u \cdots d$ is a Dyck path.
\end{definition}

\begin{figure}[htp]
    \centering
    \includegraphics[width=0.5\linewidth]{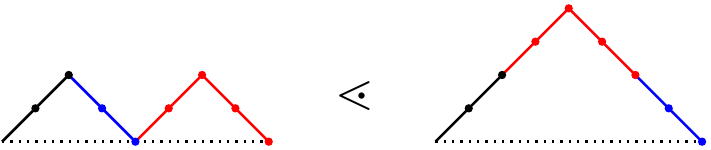}
    \caption{$u_1u_2\textcolor{blue}{d_3d_4}\textcolor{red}{u_5u_6d_7d_8} \lessdot u_1u_2\textcolor{red}{u_3u_4d_5d_6}\textcolor{blue}{d_7d_8}$ in $\mathrm{Krew}_4$.}
    \label{fig:cover_krew}
\end{figure} 

When $k=1$ in \Cref{def:krew}, we have $AudDB \lessdot AuDdB$. Similar to the Dyck-Tamari poset $\mathrm{DTam}_n$, cover relations $AudDB \lessdot AuDdB$ are cover relations that the Kreweras and Tamari lattices have in common. We call the poset induced by these cover relations the \emph{Kreweras-Tamari poset} $\mathrm{KTam}_n$.

\begin{definition}[Kreweras-Tamari poset]
\label{def:ktam}
    The \emph{Kreweras-Tamari poset} $\mathrm{KTam}_n$ is the reflexive transitive closure of cover relations $w = AudDB \lessdot w' = AuDdB$, where $w$ and $w'$ are Dyck paths of size $n$, and $D = u \cdots d$ is the Dyck path of minimal size.
\end{definition}

\begin{figure}[htp]
    \centering
    \includegraphics[width=0.5\linewidth]{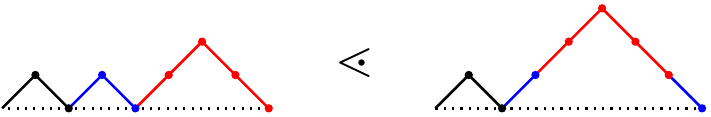}
    \caption{$u_1d_2\textcolor{blue}{u_3d_4}\textcolor{red}{u_5u_6d_7d_8} \lessdot u_1d_2\textcolor{blue}{u_3}\textcolor{red}{u_4u_5d_6d_7}\textcolor{blue}{d_8}$ in $\mathrm{KTam}_4$.}
    \label{fig:cover_ktam}
\end{figure}

As mentioned, there are $nC_n$ cover relations $AudDB \lessdot AuDdB$ in $\mathrm{KTam}_{n+1}$. Although we sketched a proof earlier, let us give a proof via a bijection between $AudDB \lessdot AuDdB$ and $AdudB \lessdot AuddB$. 

\begin{proposition}[A bijection between cover relations in the Dyck-Tamari poset and cover relations in the Kreweras-Tamari posets]
\label{prop:cov_bij}
    There exists a bijection between cover relations $AdudB \lessdot AuddB$ and cover relations $AudDB \lessdot AuDdB$. 
\end{proposition}

\begin{figure}[htp]
    \centering
    \includegraphics[width=0.8\linewidth]{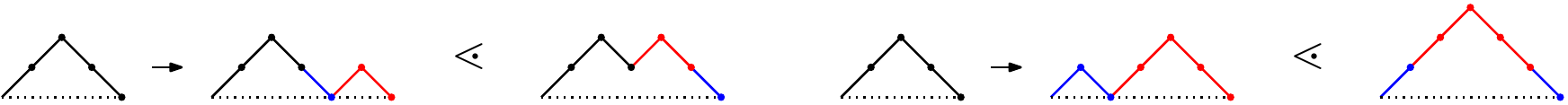}
    \caption{\textbf{Left}: $u_1u_2d_3d_4, u_1u_2d_3\textcolor{blue}{d_4}\textcolor{red}{u_5d_6} \lessdot u_1u_2d_3\textcolor{red}{u_4d_5}\textcolor{blue}{d_6}$ in $\mathrm{DTam}_3$. \textbf{Right}: $u_1u_2d_3d_4, \textcolor{blue}{u_1d_2}\textcolor{red}{u_3u_4d_5d_6} \lessdot \textcolor{blue}{u_1}\textcolor{red}{u_2u_3d_4d_5}\textcolor{blue}{d_6}$ in $\mathrm{KTam}_3$.}
    \label{fig:cov_bij}
\end{figure}

\begin{proof}
    For any Dyck path $w$, there are $n$ pairs of letters $u$ and $d$ as described in the bijection from \Cref{cor:bij}. Consider such a pair of letters $u$ and $d$. For $u$, we write $AuduB$. For $d$, we write $AdudB$. The second $u$, or the letter $u$ as written in $AuduB$ with a larger index, is paired with a letter $d$. 

    The pair of letters $u$ and $d$ defines a Dyck path as described in \Cref{cor:bij}. We rewrite $AuduB$ as $AudDB$, where $D$ is the Dyck path of minimal size following $d$ of $AudDB$. The cover relations $AdudB \lessdot AuddB$ and $AudDB \lessdot AuDdB$ are injective maps as noted in \Cref{prop:num_cov}. We have a bijection as desired, and there are $nC_n$ cover relations $AudDB \lessdot AuDdB$ by \Cref{prop:num_cov}.
\end{proof}

For reference, the figure below depicts the Dyck, Kreweras, and Tamari lattices with the Dyck-Tamari and Kreweras-Tamari posets as induced posets of cover relations in common between these lattices.

\begin{figure}[htp]
    \centering
    \includegraphics[width=0.8\linewidth]{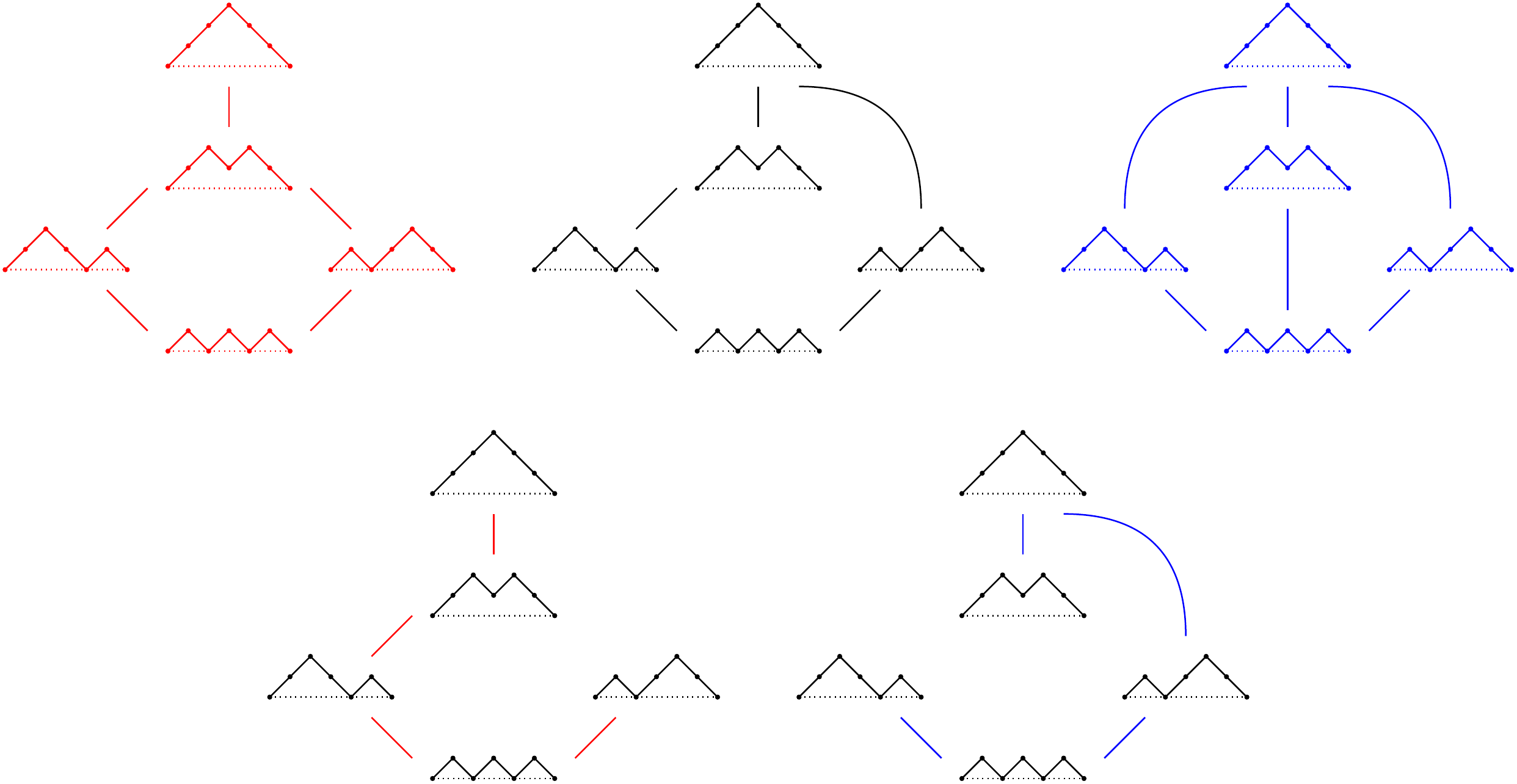}
    \caption{\textbf{Top}: The \textcolor{red}{Dyck lattice} $\textcolor{red}{\mathrm{Dyck}_3}$ in \textcolor{red}{red} (leftmost), the Tamari lattice $\mathrm{Tam}_3$ (middle), the \textcolor{blue}{Kreweras lattice} $\textcolor{blue}{\mathrm{KTam}_3}$ (rightmost). \textbf{Bottom}: The \textcolor{red}{Dyck-Tamari poset} $\textcolor{red}{\mathrm{DTam}_3}$ (left) induced by \textcolor{red}{cover relations} in \textcolor{red}{red} and the \textcolor{blue}{Kreweras-Tamari poset} $\textcolor{blue}{\mathrm{KTam}_3}$ (right) induced by \textcolor{blue}{cover relations} in \textcolor{blue}{blue}.}
    \label{fig:cat}
\end{figure}

\subsection{Properties of the Kreweras-Tamari poset}
\label{subsec:prop_w_sk}

Although the Kreweras-Tamari poset $\mathrm{KTam}_n$ is not a skeletal poset in the sense of \Cref{def:sk}, it satisfies properties similar to such a poset, namely the Dyck-Tamari poset $\mathrm{DTam}_n$; the enumeration of cover relations and \Cref{def:ktam}, for example. Before continuing, we show that the Kreweras lattice is a refinement of the Tamari lattice.

\begin{proposition}[Kreweras lattice refines Tamari lattice]
\label{prop:k_refines_t}
    The Kreweras lattice $\mathrm{Krew}_n$ is a \newline refinement of the Tamari lattice $\mathrm{Tam}_n$.
\end{proposition}

\begin{proof}
    We need to show that $w \leq w'$ in $\mathrm{Krew}_n$ implies $w \leq w'$ in $\mathrm{Tam}_n$, where $w$ and $w'$ are Dyck paths of size $n$. Using \Cref{def:krew}, consider $w = Aud^k\mathfrak{D}B \lessdot w' = Au\mathfrak{D}d^kB$, where $\mathfrak{D} = D$, the Dyck path of minimal size. For any such cover relations $w \leq w'$ in $\mathrm{Krew}_n$, there exists a chain $Aud^kDB \lessdot Aud^{k-1}DdB \lessdot Aud^{k-2}Dd^2B \lessdot \dots \lessdot AuDd^kB$ in $\mathrm{Tam}_n$ by \Cref{def:tam}. For $\mathfrak{D} \neq D$, where $D$ is as described previously, observe that $\mathfrak{D} = DE$, where $D$ is the Dyck path of minimal size with $w_1 = u$. In general, $\mathfrak{D}$ can be decomposed into at most $n$ Dyck paths $D$, i.e., $\mathfrak{D} = D_1D_2 \cdots D_k$ for $1 \leq k \leq n$, where each $D$ is of minimal size following the last letter $d$ of the preceding $D$, except the first as described earlier. Using the same argument with the chain in $\mathrm{Tam}_n$ applied to each $D$, it follows that $w \leq w'$ in $\mathrm{Krew}_n$ implies $w \leq w'$ in $\mathrm{Tam}_n$. We conclude that the Kreweras lattice $\mathrm{Krew}_n$ is a refinement of the Tamari lattice $\mathrm{Tam}_n$. 
\end{proof}

Our goal is to show that the Hasse diagram of $\mathrm{KTam}_n$ is the largest spanning subgraph of the Hasse diagrams of $\mathrm{Tam}_n$ and $\mathrm{Krew}_n$, which we proof using bracket vectors $v$ and the height function $h$. 

\begin{proposition}[Height of intervals in the Kreweras-Tamari poset]
\label{prop:h_intervals}
    For comparable elements $x$ and $y$ in the Kreweras-Tamari poset $\mathrm{KTam}_n$, let $v = (v_1,v_2, \dots, v_n)$ and $v' = (v_1',v_2', \dots, v_n')$ be bracket vectors of $x$ and $y$, respectively. Consider $k$ and $k'$ the numbers of $v_i \neq 0$ and $v_i' \neq 0$, respectively, for $1 \leq i \leq n$. The height of an interval $[x,y]$ in $\mathrm{KTam}_n$ is $h([x,y]) = k' - k$.
\end{proposition}

\begin{proof}
    Consider cover relations $w = AudDB \lessdot w' = AuDdB$ in $\mathrm{KTam}_n$. Since $w \lessdot w'$ is also in $\mathrm{Tam}_n$, by \Cref{lem:cover_bracket}, we write $w \lessdot w'$ in terms of bracket vectors; explicitly, $v = (v_1, v_2, \dots, v_i, \dots, v_n) \lessdot v' = (v_1, v_2, \dots, v_i + v_{(i+v_i+1)} + 1, v_n)$. Recall from the proof of \Cref{lem:cover_bracket} that $v_i$ is the number of pairs of letters $u$ and $d$ between the $i$-th pair of letters $u$ and $d$ with $d$ as written in $AudDB$. In this particular instance, $d$ is paired with the letter $u$ preceding it as written in $AudDB$. So, $v_i = 0$. We have $v = (v_1, v_2, \dots, v_i = 0, \dots, v_n) \lessdot v' = (v_1, v_2, \dots, v_i' = v_{i+1} + 1, \dots, v_n)$. The difference $k' - k$ of the nonzero $v_i$ is $1$. Since all cover relations in $\mathrm{KTam}_n$ are of the form $AudDB \lessdot AuDdB$, by transitivity, for any interval $[x,y]$ in $\mathrm{KTam}_n$, $h([x,y]) = k' - k$. 
\end{proof}

\begin{corollary}[Kreweras-Tamari poset as spanning subgraph]
\label{cor:span}
    The Hasse diagram of $\mathrm{KTam}_n$ is the largest spanning (directed) subgraph of the Hasse diagrams of $\mathrm{Tam}_n$ and $\mathrm{Krew}_n$.
\end{corollary}

\begin{proof}
    It suffices to prove that the Hasse diagram of $\mathrm{KTam}_n$ is connected. For an element $x$ in $\mathrm{KTam}_n$ with bracket vector $v$, there exists cover relations $v \lessdot v'$ and/or $v' \lessdot v$, where $v'$ has $1$ more or $1$ less nonzero $v_i'$ than $v$. By transitivity, there exists a path from an element to any other element in the Hasse diagram of $\mathrm{KTam}_n$. It follows that the Hasse diagram of $\mathrm{KTam}_n$ is connected.
\end{proof}

In summary, the Dyck-Tamari poset $\mathrm{DTam}_n$ and the Kreweras-Tamari poset $\mathrm{KTam}_n$ are the largest spanning subgraphs for their respective pairs of lattices. In this setting, we consider $\mathrm{KTam}_n$ a skeletal poset. Since $h(\mathrm{KTam}_n) = n < h(\mathrm{DTam}_n) = \frac{n(n-1)}{2}$ with more elements $x$ in $\mathrm{KTam}_n$ such that $h(x_\downarrow)$ is small compared to $h(x_\downarrow)$ of $x$ in the Dyck-Tamari poset, we also call the Kreweras-Tamari poset the \emph{wide skeletal poset} of the Tamari lattice. In turn, we call the Dyck-Tamari poset the \emph{long skeletal poset} of the Tamari lattice. 

\subsection{Intervals in the Tamari lattice}
\label{subsec:intervals_tam}

Besides computing $h$ in $\mathrm{DTam}_n$ and $\mathrm{KTam}_n$, we use bracket vectors and Dyck paths to count intervals in these posets, which are intervals in the Tamari lattice $\mathrm{Tam}_n$. We start with characterizing intervals in $\mathrm{Tam}_n$ that we separated into cases.

\begin{proposition}[Characterizing the intervals in the Tamari lattice, I]
\label{prop:simple_int}
    Let $[v = (v_1,v_2, \dots, v_n), \newline v' = (v_1', v_2', \dots, v_n')]$ be an interval in the Tamari lattice $\mathrm{Tam}_n$ with bracket vectors $v$ and $v'$. If \newline $v_1 = v_1' = k$ for $0 \leq k \leq n-1$, then $[v,v']$ can be written as $[uLdR, uL'dR']$, where $L, L', R$ and \newline $R'$ are Dyck paths, and $[L, L']$ and $[R, R']$ are intervals in $\mathrm{Tam}_k$ and $\mathrm{Tam}_{n-k-1}$, respectively.
\end{proposition}

\begin{figure}[htp]
    \centering
    \includegraphics[width=0.225\linewidth]{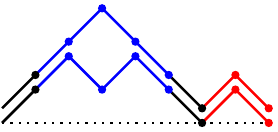}
    \caption{The interval $[(2, \textcolor{blue}{0}, \textcolor{blue}{0}, \textcolor{red}{0}), (2, \textcolor{blue}{1}, \textcolor{blue}{0}, \textcolor{red}{0})]$ in $\mathrm{Tam}_4$, and intervals $[(\textcolor{blue}{0}, \textcolor{blue}{0}), (\textcolor{blue}{1}, \textcolor{blue}{0})]$ and $[(\textcolor{red}{0}), (\textcolor{red}{0})]$ in $\textcolor{blue}{\mathrm{Tam}_2}$ and $\textcolor{red}{\mathrm{Tam}_1}$, respectively.}
    \label{fig:simple_int}
\end{figure}

\begin{proof}
    Consider an interval $[v = (v_1,v_2, \dots, v_n), v' = (v_1', v_2', \dots, v_n')]$ in the Tamari lattice $\mathrm{Tam}_n$ such that $v_1 = v_1' = k$. Recall from \Cref{cor:bij} that $v_i$ for $1 \leq i \leq n$ is the number of pairs of letters $u$ and $d$ between the $i$-th pair of letters $u$ and $d$ in Dyck paths $w$ of $v$ and $w'$ of $v'$. By \Cref{def:bracket}, $(v_2, v_3, \dots, v_{k+1})$, $(v_2', v_3', \dots, v_{k+1}')$, $(v_{k+2}, v_{k+3}, \dots, v_n)$, and $(v_{k+2}', v_{k+3}', \dots, v_n')$ are bracket vectors. Since $[v, v']$ is an interval in $\mathrm{Tam}_n$, $[(v_2, v_3, \dots, v_{k+1}), (v_2', v_3', \dots, v_{k+1}')]$ and $[(v_{k+2}, v_{k+3}, \dots, v_n), (v_{k+2}', v_{k+3}', \dots, v_n')]$ are intervals in $\mathrm{Tam}_k$ and $\mathrm{Tam}_{n-k-1}$, respectively. 
    
    $[(v_2, v_3, \dots, v_{k+1}), (v_2', v_3', \dots, v_{k+1}')]$ and $[(v_{k+2}, v_{k+3}, \dots, v_n), (v_{k+2}', v_{k+3}', \dots, v_n')]$ can be written as $[L, L']$ and $[R, R']$, respectively, where $L$ and $L$' are Dyck paths of size $k$, and $R$ and $R'$ are Dyck paths of size $n-k-1$. Thus, if $v_1 = v_1'$, then we can write $[v, v']$ as $[uLdR, uL'dR']$.
\end{proof}

\begin{proposition}[Characterizing the intervals in the Tamari lattice, II]
\label{prop:left_int}
    Let $[v = (v_1,v_2, \dots, v_n), \newline v' = (v_1', v_2', \dots, v_n')]$ be an interval in the Tamari lattice $\mathrm{Tam}_n$ with bracket vectors $v$ and $v'$. If $v_1 < v_1' = n-1$, then $[v, v']$ can be written as $[uLMdT, uL'M'd]$, where $L, L', M, M'$ and $W$ are Dyck paths, and $[L, L']$ and $[MT, M']$ are intervals in $\mathrm{Tam}_{k}$ and $\mathrm{Tam}_{n-k-1}$, respectively, for $0 \leq k \leq v_1$, the largest value $k$ such that $[L, L']$ is an interval.
\end{proposition}

\begin{figure}[htp]
    \centering
    \includegraphics[width=0.225\linewidth]{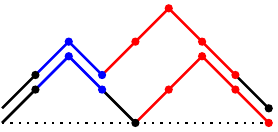}
    \caption{The interval $[(1, \textcolor{blue}{0}, \textcolor{red}{1}, \textcolor{red}{0}), (3, \textcolor{blue}{0}, \textcolor{red}{1}, \textcolor{red}{0})]$ in $\mathrm{Tam}_4$, and intervals $[\textcolor{blue}{0}, \textcolor{blue}{0}]$ and $[(\textcolor{red}{1}, \textcolor{red}{0}), (\textcolor{red}{1}, \textcolor{red}{0})]$ in $\textcolor{blue}{\mathrm{Tam}_1}$ and $\textcolor{red}{\mathrm{Tam}_2}$, respectively.}
    \label{fig:left_int}
\end{figure}

\begin{proof}
    Consider an interval $[v = (v_1, v_2, \dots, v_n), v' = (v_1', v_2', \dots, v_n')]$ in the Tamari lattice $\mathrm{Tam}_n$ such that $v_1 < v_1' = n-1$. Recall from \Cref{cor:bij} that $v_i$ for $1 \leq i \leq n$ counts the number of pairs of letters $u$ and $d$ between the $i$-th pair of letters $u$ and $d$ in Dyck paths $w$ of $v$ and $w'$ of $v'$. Let $k \leq v_1$. By \Cref{def:bracket}, $(v_2, v_3, \dots, v_{k+1})$, $(v_2', v_3', \dots, v_{k+1}')$, $(v_{k+2}, v_{k+3}, \dots, v_n)$, and $(v_{k+2}', v_{k+3}', \dots, v_n')$ are bracket vectors. We fix $k$ to be the largest value such that $[(v_2, v_3, \dots, v_{k+1}), \newline (v_2', v_3', \dots, v_{k+1}')]$ is an interval in $\mathrm{Tam}_{k}$. Since $[v, v']$ is an interval in $\mathrm{Tam}_n$, $[(v_{k+2}, v_{k+3}, \dots, v_n), \newline (v_{k+2}', v_{k+3}', \dots, v_n')]$ is an interval in $\mathrm{Tam}_{n-k-1}$. 

    $[(v_2, v_3, \dots, v_{k+1}), (v_2', v_3', \dots, v_{k+1}')]$ and $[(v_{k+2}, v_{k+3}, \dots, v_n), (v_{k+2}', v_{k+3}', \dots, v_n')]$ can be written as $[L, L']$ and $[MT, M']$, respectively, where $L$ and $L'$ are Dyck paths of size $k$, and $MT$ and $M'$ are Dyck paths of size $n-k-1$, where $M$ can be written as $(v_{k+2}, v_{k+3}, \dots, v_{v_1+1})$ and $T$ can be written as $(v_{v_1+2}, v_{v_1+3}, \dots, v_{v_1'+1 = n})$. Thus, if $v_1 < v_1' = n-1$, we can write $[v, v']$ as $[uLMdT, uL'M'd]$ (see \Cref{fig:all_int}).
\end{proof}

\begin{lemma}[Characterizing the intervals in the Tamari lattice, III]
\label{lem:all_int}
    Let $[v = (v_1,v_2, \dots, v_n), \newline v' = (v_1', v_2', \dots, v_n')]$ be an interval in the Tamari lattice $\mathrm{Tam}_n$ with bracket vectors $v$ and $v'$. Any interval $[v, v']$ can be written as $[uLMdTR, uL'M'dR']$, where $L, L', M, M', R, R'$ and $W$ are Dyck paths, and $[L, L']$, $[MT, M']$, and $[R, R']$ are intervals in $\mathrm{Tam}_{k}$, $\mathrm{Tam}_{k''}$, and $\mathrm{Tam}_{n-k-k''-1}$, respectively, for $0 \leq k \leq v_1$, the largest value $k$ such that $[L, L']$ is an interval, and $k'' = v_1' - k$.
\end{lemma}

\begin{figure}[htp]
    \centering
    \includegraphics[width=0.35\linewidth]{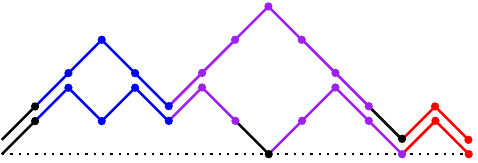}
    \caption{The interval $[(3, \textcolor{blue}{0}, \textcolor{blue}{0}, \textcolor{purple}{0}, \textcolor{purple}{1}, \textcolor{purple}{0}, \textcolor{red}{0}), (5, \textcolor{blue}{1}, \textcolor{blue}{0}, \textcolor{purple}{2}, \textcolor{purple}{1}, \textcolor{purple}{0}, \textcolor{red}{0})]$ in $\mathrm{Tam}_7$, and intervals $[(\textcolor{blue}{0}, \textcolor{blue}{0}), (\textcolor{blue}{1}, \textcolor{blue}{0})]$, $[(\textcolor{purple}{0}, \textcolor{purple}{1}, \textcolor{purple}{0}), (\textcolor{purple}{2}, \textcolor{purple}{1}, \textcolor{purple}{0})]$ and $[(\textcolor{red}{0}), (\textcolor{red}{0})]$ in $\textcolor{blue}{\mathrm{Tam}_2}$, $\textcolor{purple}{\mathrm{Tam}_3}$, and $\textcolor{red}{\mathrm{Tam}_1}$, respectively.}
    \label{fig:all_int}
\end{figure}

\begin{proof}
    The lemma follows from \Cref{prop:simple_int} and \Cref{prop:left_int}. It reduces to \Cref{prop:simple_int} if $[MT, M']$ is the empty interval, and to \Cref{prop:left_int} if $[R, R']$ is the empty interval. 
\end{proof}

Using \Cref{lem:all_int}, we can characterize intervals in the Dyck-Tamari poset $\mathrm{DTam}_n$ and in the Kreweras-Tamari poset $\mathrm{KTam}_n$. Then, the descriptions of these intervals allow us to count them; see \Cref{subsec:intervals_dtam} for intervals in $\mathrm{DTam}_n$, which are counted by the number of ternary trees (\href{https://oeis.org/A001764}{A001764}, OEIS~\cite{OEIS}), and \Cref{subsec:intervals_ktam} for intervals in $\mathrm{KTam}_n$, which are counted by the number of hex trees (\href{https://oeis.org/A002212}{A002212}, OEIS~\cite{OEIS}), or ternary trees with restrictions on the outdegrees of vertices.

\subsection{Counting intervals in the Dyck-Tamari poset}
\label{subsec:intervals_dtam}

We begin with the characterization of the intervals in $\mathrm{DTam}_n$ using \Cref{lem:all_int}. We then study intervals of the form $[uMdT, uM'd]$, i.e., $[uLMdTR, uL'M'dR']$, where Dyck paths $L, L', R$, and $R'$ are empty, and enumerate them. From there on, it becomes a straightforward exercise to count the number of intervals in $\mathrm{DTam}_n$. We recommend that the reader see \Cref{fig:cat_bij} and the figures in this subsection often. 

\begin{lemma}[Characterizing the intervals in the Dyck-Tamari poset]
\label{lem:dtam_int}
    Consider an interval $[uLMdTR, uL'M'dR']$, as described in \Cref{lem:all_int}, in the Dyck-Tamari poset $\mathrm{DTam}_n$; then $T = udud \cdots ud$, i.e., $T \neq u \cdots uu \cdots d$.
\end{lemma}

\begin{figure}[htp]
    \centering
    \captionsetup{width = 0.905\linewidth}
    \includegraphics[width=0.35\linewidth]{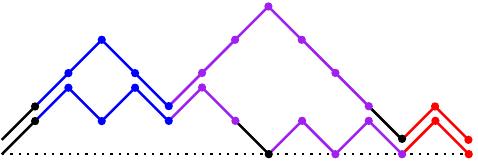}
    \caption{The interval $[(3, \textcolor{blue}{0}, \textcolor{blue}{0}, \textcolor{purple}{0}, \textcolor{purple}{0}, \textcolor{purple}{0}, \textcolor{red}{0}), (5, \textcolor{blue}{1}, \textcolor{blue}{0}, \textcolor{purple}{2}, \textcolor{purple}{1}, \textcolor{purple}{0}, \textcolor{red}{0})]$ in $\mathrm{DTam}_7$, and intervals $[(\textcolor{blue}{0}, \textcolor{blue}{0}), (\textcolor{blue}{1}, \textcolor{blue}{0})]$, $[(\textcolor{purple}{0}, \textcolor{purple}{0}, \textcolor{purple}{0}), (\textcolor{purple}{2}, \textcolor{purple}{1}, \textcolor{purple}{0})]$ and $[(\textcolor{red}{0}), (\textcolor{red}{0})]$ in $\textcolor{blue}{\mathrm{DTam}_2}$, $\textcolor{purple}{\mathrm{DTam}_3}$, and $\textcolor{red}{\mathrm{DTam}_1}$, respectively.}
    \label{fig:dtam_int}
\end{figure}

\begin{proof}
    By \Cref{lem:all_int}, we can write any interval $[w, w']$ in $\mathrm{DTam}_n$ as $[uLMdTR, uL'M'dR']$. Recall from \Cref{def:dtam} that cover relations in $\mathrm{DTam}_n$ are of the form $AdudB \lessdot AuddB$. Observe that there exists cover relations $uLMdTR = uLMdudR \lessdot uL'M'dR' = uLMTdR = uLMuddR$. So, $T \neq u \cdots duu \cdots d$. Similarly, $T \neq uu \cdots d$ since $uLMdTR = uLMduuddR \not\leq uL'MTdR' = uL'MuudddR'$. We conclude that $T = udud \cdots ud$, i.e., $T \neq u \cdots uu \cdots d$.  
\end{proof}

\begin{corollary}[Comparability of elements in the Dyck-Tamari poset, I]
\label{cor:comp}
    Let $[w, w']$ be an interval in $\mathrm{DTam}_n$. Either $w \neq u \cdots duu \cdots d$ or if there exists any $d_{k}u_{k+1}u_{k+2}$ for $2 \leq k \leq 2n-4$ where $w = u \cdots d_ku_{k+1}u_{k+2} \cdots d$, then $w' = u \cdots d_{k}u_{k+1}u_{k+2} \cdots d$; equivalently, $d_ku_{k+1}u_{k+2}$ is fixed.
\end{corollary}

\begin{figure}[htp]
    \centering
    \includegraphics[width=0.22\linewidth]{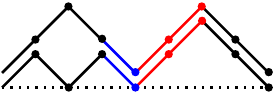}
    \caption{The interval $[u_1d_2u_3\textcolor{blue}{d_4}\textcolor{red}{u_5u_6}d_7d_8, u_1u_2d_3\textcolor{blue}{d_4}\textcolor{red}{u_5u_6}d_7d_8]$ in $\mathrm{DTam}_4$ where $\textcolor{blue}{d_4}\textcolor{red}{u_5u_6}$ is fixed. In terms of bracket vectors, the interval is $[(0, 0, \textcolor{red}{1}, 0), (1, 0, \textcolor{red}{1}, 0)]$.}
    \label{fig:fixed}
\end{figure}

\Cref{cor:comp} is a comparability criterion for elements in $\mathrm{DTam}_n$. Since it is stated in terms of Dyck paths, we restate this criterion in terms of bracket vectors. Both versions are necessary for us to enumerate the intervals in $\mathrm{DTam}_n$.

\begin{corollary}[Comparability of elements in the Dyck-Tamari poset, II]
\label{cor:comp_bracket}
    Let $[v = (v_1, v_2, \dots, v_n), \newline v' =(v_1', v_2', \dots, v_n')]$ be an interval in $\mathrm{DTam}_n$. Consider $v_i \geq 1$ for $2 \leq i \leq n-1$, and $v_{i-1} = 0$. Either such a $v_i$ is not in $v$ or if there exists such a $v_i$, then $v_{i-1} = 0$ and $v_i = v_i'$. 
\end{corollary}

\begin{proof}
    It suffices to use the bijection from \Cref{cor:bij}. Recall from \Cref{cor:comp} the condition that $d_{k}u_{k+1}u_{k+2}$ in $w$ for $2 \leq k \leq 2n-4$. We simply need to state this condition in terms of bracket vectors. So, let us write $w$ as a bracket vector $v$. Recall that each $v_i$ for $1 \leq i \leq n$ in $v$ counts the number of pairs of letters $u$ and $d$ between the $i$-th pair letters $u$ and $d$ in $w$. If $d_{k}u_{k+1}u_{k+2}$ is in $w$, then the $v_i$ corresponding to $u_{k+1}$ satisfies $v_i \geq 1$, and $v_{i-1} = 0$ by \Cref{def:bracket}. So, if $v_i \geq 1$, then $v_i = v_i'$ (see \Cref{fig:fixed}). 
\end{proof}

With \Cref{cor:comp} and \Cref{cor:comp_bracket}, we can map intervals $[w, w']$ in $\mathrm{DTam}_{n-1}$ to and from intervals $[uMdT, uM'd]$ in $\mathrm{DTam}_n$ as follows.

\begin{corollary}[Mapping intervals in the Dyck-Tamari poset, I]
\label{cor:dtam_map_int}
    Consider the intervals $[w, w']$ in $\mathrm{DTam}_{n-1}$. $[wud, uw'd]$ or $[uMdT, uM'd]$ are intervals in $\mathrm{DTam}_{n}$ if and only if $w \neq u \cdots duu \cdots d$. Conversely, consider the intervals $[uMdT, uM'd]$ or $[(v_1, v_2, \dots, v_{n-1}, 0), (n-1, v_1', v_2', \dots, v'_{n-1})]$ in $\mathrm{DTam}_n$. $[(v_1, v_2, \dots, v_{n-1}), (v_1', v_2', \dots, v_{n-1})]$ are intervals in $\mathrm{DTam}_{n-1}$ if and only if $v_i \leq v_i'$ for $1 \leq i \leq n-1$ and there does not exist $v_j \geq 1$ for $2 \leq j \leq n-2$ such that $v_j = v_j'$. 
\end{corollary}

\begin{figure}[htp]
    \centering
    \includegraphics[width=0.5\linewidth]{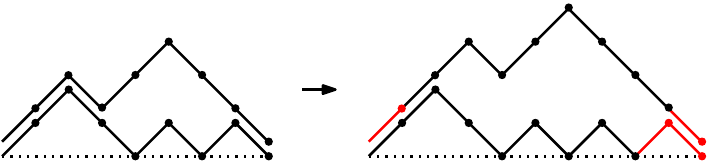}
    \caption{The interval $[(1, 0, 0, 0), (3, 0, 1, 0)]$ in $\mathrm{DTam}_4$ mapped to the interval $[(1, 0, 0, 0, \textcolor{red}{0}), (\textcolor{red}{4}, 3, 0, 1, 0)]$ in $\mathrm{DTam}_5$.}
    \label{fig:map_1}
\end{figure}

Explicitly, for intervals in $[uLMdT, uL'M'dR']$ in $\mathrm{DTam}_{n-1}$, we map these intervals to $[uLMdTud, \newline uuL'M'dR'd]$ that we rewrite as $[uMdT, uM'd]$. The intervals $[uMdT, uM'd]$ are in $\mathrm{DTam}_n$ given that they satisfy the condition in \Cref{cor:dtam_map_int}. We are left with intervals $[w, w']$ in $\mathrm{DTam}_{n-1}$ where $w = u \cdots d_ku_{k+1}u_{k+2} \cdots d$ for some $d_ku_{k+1}u_{k+2}$ where $2 \leq k \leq 2(n-1)-4$ and intervals $[uMdT, uM'd] = [wud, uw'd]$ described in the following corollary.

\begin{corollary}[Mapping intervals in the Dyck-Tamari poset, II]
\label{cor:dtam_map_int_2}
    Let $[(v_1, v_2, \dots, v_{n-1}, 0), (n-1, v_1', v_2', \dots, v'_{n-1})]$ be intervals in $\mathrm{DTam}_n$. If $v_1 \neq 0$ and there exists $v_{i+1} = v_i'$ for $1 \leq i \leq v_1$, then $(v_1, v_2, \dots, v_{n-1}) \not\leq (v_1', v_2', \dots, v_{n-1}')$.
\end{corollary}

\begin{figure}[htp]
    \centering
    \captionsetup{width = 0.65\linewidth}
    \includegraphics[width=0.225\linewidth]{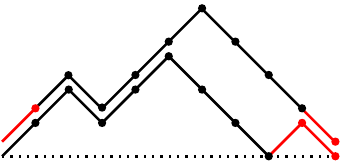}
    \caption{The interval $[(3, 0, 1, 0, \textcolor{red}{0}), (\textcolor{red}{4}, 0, 2, 1, 0)]$ in $\mathrm{DTam}_5$. Observe that $(3, 0, 1, 0) \not\leq (0, 2, 1, 0)$.}
    \label{fig:duu}
\end{figure}

Observe that intervals $[(v_1, v_2, \dots, v_{n-1}, 0), (n-1, v_1', v_2', \dots, v'_{n-1})]$ includes those of the form $[uLMdT, uLM'd]$ where the interval $[L, L']$ is nonempty. Specifically, if $v_2 = v_1'$ in \Cref{cor:dtam_map_int_2}, then we have the interval $[(v_2, v_3, \dots, v_{v_1}), (v_1', v_2', \dots, v_{v_1})]$ in $\mathrm{DTam}_{v_1}$. We only consider intervals $[uMdT, uM'd]$ in $\mathrm{DTam}_n$. So, for $i$ in \Cref{cor:dtam_map_int_2}, we ignore the case $i = 1$. Consider $i$, the smallest index such that $v_{i} = 0$ for $3 \leq i \leq v_1 + 1$ and $v_1 \geq 2$; then for any $v_j$ such that $j < i$, $v_j > 0$. So, we can write $[uMdT, uM'd]$ as $u^k = u_1u_2 \cdots u_k$ for $2 \leq k \leq n-2$ and $[D_{\max} , D'_{\max}]$ is an interval in $\mathrm{DTam}_{m}$ for $1 \leq m \leq v_2$ (see the rightmost interval in \Cref{fig:dtam_bij}. We use the subscript $\max$ to indicate that we take Dyck paths of maximal size following $u^k$ in $[u^kD_{\max} E, u^kD'_{\max} E']$ such that $[D_{\max}, D'_{\max}]$ is an interval in $\mathrm{DTam}_m$. 

With this in mind, consider Dyck paths $w = CduuE$. Let $d$ be the last letter of the Dyck path $D_{\max}$ of maximal size as written in $w$. Observe that $w = C'D_{max}uuE$, where $Cd = C'D_{\max}$. Note the similarity between $w = C'D_{max}uuE$ and $u^kD_{\max}E$. Using this observation, we count intervals $[uMdT, uM'd]$ in $\mathrm{DTam}_n$; informally, by adding/removing a pair of letters $u$ and $d$, and swapping $D_{\max}uu$ or $uuD_{\max}$ with the other (see \Cref{fig:dtam_bij}). Then, we count all intervals in $\mathrm{DTam}_n$.

\begin{lemma}[Counting intervals in the Dyck-Tamari poset, I]
\label{lem:dtam_hard}
    Let $|\mathrm{DTam}_n|$ be the number of \newline intervals $[w, w']$, or $[v, v']$, in the Dyck-Tamari poset $\mathrm{DTam}_n$. The number of intervals $[uMdT, uM'd]$ in $\mathrm{DTam}_n$ is $|\mathrm{DTam}_{n-1}|$.
\end{lemma}

\begin{figure}[htp]
    \centering
    \includegraphics[width=0.8\linewidth]{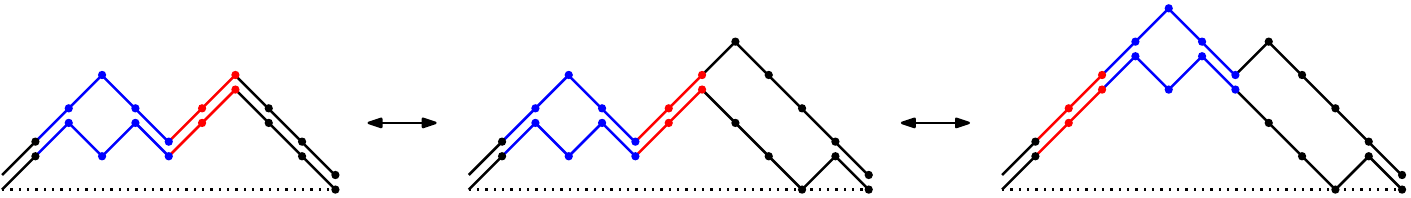}
    \caption{The bijection in the proof of \Cref{lem:dtam_hard}.}
    \label{fig:dtam_bij}
\end{figure}

\begin{proof} 
    Before moving forward, we recommend that the reader keep in mind what was previously said right after \Cref{cor:dtam_map_int} and after \Cref{cor:dtam_map_int_2}. 

    By \Cref{cor:dtam_map_int}, it suffices to consider intervals $[w, w'] = [u \dots duuE, u \dots duuE']$ in $\mathrm{DTam}_{n-1}$ and intervals $[uMdT, uM'd]$, where $u^k = u_1u_2 \cdots u_k$ for $2 \leq k \leq n-2$ and $[D_{\max}, D'_{\max}]$ is the largest interval in $\mathrm{DTam}_{m}$ for $1 \leq m \leq v_2$ following $u^k$. Recall that the intervals $[uMdT, uM'd]$ here are those such that $uMdT'  \not\leq M'$ where $T'ud = T$.

    Observe that intervals $[w, w']$ can be written as $[CD_{\max}uuE, C'D_{\max}'uuE']$, where $uu$ in $[w, w']$ have the same indices and $[D_{\max}, D_{\max}]$ is the largest interval, i.e., the sizes of the Dyck paths $D_{\max}$ and $D'_{\max}$ are maximal, preceding $uu$. Since $D_{\max}$ and $D'_{\max}$ are maximal, $C = C' = u^{k-2} = u_1u_2 \cdots u_{k-2}$ for $2 \leq k \leq n-3$. It follows that $[u^{k-2}D_{\max}uuE, u^{k-2}D'_{\max}uuE']$. We want to map intervals $[u^{k-2}D_{\max}uuE, u^{k-2}D'_{\max}uuE']$ in $\mathrm{DTam}_{n-1}$ to $[uMdT, uM'd]$ in $\mathrm{DTam}_n$. So, consider the following bijections:
    \begin{align*}
        [u^{k-2}D_{\max}uuE, u^{k-2}D'_{\max}uuE'] &\leftrightarrow [u^{k-2}D_{\max}uuEud, u^{k-2}D_{\max}uuuE'd] \\
        &\leftrightarrow [u^kD_{\max}Eud, u^{k}D_{\max}uE'd].
    \end{align*}
    We have $[u^kD_{\max}Eud, u^{k}D_{\max}uE'd] = [uMdT, uM'd]$ in $\mathrm{DTam}_n$ as desired. For the backward direction, it suffices to show that $[uMdT, uM'd]$ can be written as $[u^kD_{\max}Eud, u^{k}D_{\max}uE'd]$. Recall that $T = udud \cdots ud$. So, we can write $uMdT$ as $u^kD_{\max}Eud$. Note that $uM'd \neq u^nd^n = u_1u_2 \cdots u_nd_{n+1}d_{n+2} \cdots d_{2n}$ since $u^nd^n \geq uMdT'$ for any Dyck path $uMdT'$ of size $n-1$. We write $uM'd$ as $u^{k}D_{\max}uE'd$. So, $[uMdT, uM'd] = [u^kD_{\max}Eud, u^{k}D_{\max}uE'd]$ as desired. We conclude that there are $|\mathrm{DTam}_{n-1}|$ intervals $[uMdT, uM'd]$. 
\end{proof}

\begin{theorem}[Counting intervals in the Dyck-Tamari poset, II]
\label{thm:count_dtam}
    Let $|\mathrm{DTam}_n|$ be the number of intervals $[w, w']$, or $[v, v']$, in the Dyck-Tamari poset $\mathrm{DTam}_n$. The number of intervals in $\mathrm{DTam}_n$ is the number of ternary trees (\href{https://oeis.org/A001764}{A001764}, OEIS~\cite{OEIS}), i.e.,   
    \begin{align*}
        |\mathrm{DTam}_0| = 1, |\mathrm{DTam}_n| = \sum_{i,j=0}^{n-1} |\mathrm{DTam}_i||\mathrm{DTam}_j||\mathrm{DTam}_{n-i-j-1}|.
    \end{align*}
\end{theorem}

\begin{proof}
    The empty interval is unique, i.e., $|\mathrm{DTam}_0| = 1$. Consider intervals $[uLMdTR, uL'M'dR']$ in $\mathrm{DTam}_n$. Let $L$ and $L'$ be Dyck paths of length $i$, and $MT$ and $M'$ be Dyck paths of length $j$, for $0 \leq i + j \leq n-1$. So, $R$ and $R'$ are Dyck paths of length $n-i-j-1$. We have $|\mathrm{DTam}_i|$ intervals $[L, L']$, $|\mathrm{DTam}_j|$ intervals $[uMdT, uM'd]$ by \Cref{lem:dtam_hard}, and $|\mathrm{DTam}_{n-i-j-1}|$ intervals $[R, R']$ for any fixed $i$ and $j$. So, the equality $|\mathrm{DTam}_n| = \sum_{i,j=0}^{n-1} |\mathrm{DTam}_i||\mathrm{DTam}_j||\mathrm{DTam}_{n-i-j-1}|$ holds as desired.
\end{proof}

\subsection{Counting intervals in the Kreweras-Tamari poset}
\label{subsec:intervals_ktam}

Using \Cref{lem:all_int}, we characterize the intervals in $\mathrm{KTam}_n$, which then allow us to count these intervals directly. We highlight here that intervals in $\mathrm{KTam}_n$ have some similarities with intervals in $\mathrm{DTam}_n$; notably, the enumerations of intervals in $\mathrm{DTam}_n$ and in $\mathrm{KTam}_n$.

\begin{lemma}[Characterizing the intervals in the Kreweras-Tamari poset]
\label{lem:ktam_int}
    Consider an interval $[uLMdTR, uL'M'dR']$, as described in \Cref{lem:all_int}, in the Kreweras-Tamari poset $\mathrm{KTam}_n$; then either the Dyck paths $M, M'$ and $T$ are empty or the Dyck paths $L, L'$ and $M$ are empty. 
\end{lemma}

\begin{figure}[htp]
    \centering
    \includegraphics[width=0.225\linewidth]{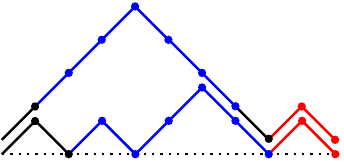}
    \caption{The interval $[(0, \textcolor{blue}{0}, \textcolor{blue}{1}, \textcolor{blue}{0}, \textcolor{red}{0}), (3, \textcolor{blue}{2}, \textcolor{blue}{1}, \textcolor{blue}{0}, \textcolor{red}{0})]$ in $\mathrm{KTam}_5$, and intervals $[(\textcolor{blue}{0}, \textcolor{blue}{1}, \textcolor{blue}{0}), (\textcolor{blue}{2}, \textcolor{blue}{1}, \textcolor{blue}{0})]$ and $[(\textcolor{red}{0}), (\textcolor{red}{0})]$ in $\mathrm{KTam}_3$ and $\mathrm{KTam}_1$, respectively.}
    \label{fig:ktam_int}
\end{figure}

\begin{proof}
    By \Cref{prop:simple_int}, if $v_1 = v_1'$, then an interval $[v = (v_1, v_2, \dots, v_n), v' = (v_1', v_2', \dots, v_n')]$ in $\mathrm{KTam}_n$ can be written as $[uLdR, uL'dR']$. This is the case where the Dyck paths $M$, $M'$, and $T$ are empty. It suffices to show that if $v_1 < v_1'$, then the Dyck paths $L, L'$, and $M$ in $[uLMdTR, uL'M'dR']$ are empty. By \Cref{lem:cover_bracket} and \Cref{def:ktam}, $v = (v_1, v_2, \dots, v_i = 0, \dots, v_n) \lessdot v' = (v_1, v_2, \dots, v_i' = v_{i+1} + 1, v_n)$ is a cover relation in $\mathrm{KTam}_n$. Consider an interval $[v, v']$ such that, for a fixed $i$, $v_i < v_i'$. Then, by \Cref{def:ktam}, $v_i = 0$. So, if $v_1 < v_1'$, then $v_1 = 0$. It follows that the Dyck paths $L$ and $M$ are empty. Since $[L, L']$ is an interval by construction, the Dyck path $L'$ must also be empty. We have the lemma as desired.
\end{proof}

\Cref{lem:ktam_int} tells us that to count intervals in $\mathrm{KTam}_n$, it suffices to consider intervals of the form $[uLdR, uL'dR']$ and $[udTR, uM'dR']$. We count these intervals in $\mathrm{KTam}_n$ as follows.

\begin{theorem}[Counting intervals in the Kreweras-Tamari poset]
\label{thm:count_ktam}
    Let $|\mathrm{KTam}_n|$ be the number of intervals $[w, w']$, or $[v, v']$, in the Kreweras-Tamari poset $\mathrm{KTam}_n$. The number of intervals in $\mathrm{KTam}_n$ is the number of hex trees (\href{https://oeis.org/A002212}{A002212}, OEIS~\cite{OEIS}), i.e.,   
    \begin{align*}
        |\mathrm{KTam}_0| = 1, |\mathrm{KTam}_n| = |\mathrm{KTam}_{n-1}| + \sum_{i=0}^{n-1} |\mathrm{KTam}_i||\mathrm{KTam}_{n-i-1}|.
    \end{align*}
\end{theorem}

\begin{proof}
    The empty interval is unique, i.e., $|\mathrm{KTam}_0| = 1$. By \Cref{lem:ktam_int}, an interval $[w, w']$ in $\mathrm{KTam}_n$ is either of the form $[uLdR, uL'dR']$ or $[udTR, uM'dR']$. Consider intervals $[uLdR, uL'dR']$. Let $L$ and $L'$ be Dyck paths of length $i$ for $0 \leq i \leq n-1$. So, $R$ and $R'$ are Dyck paths of length $n-i-1$. We have $|\mathrm{KTam}_i|$ intervals $[L, L']$ and $|\mathrm{KTam}_{n-i-1}|$ intervals $[R, R']$ for any fixed $i$. Hence, there are $\sum_{i=0}^{n-1} |\mathrm{KTam}_{i}||\mathrm{KTam}_{n-i-1}|$ intervals $[uLdR, uL'dR']$. 

    Consider the intervals $[udTR, uM'dR']$ in $\mathrm{KTam}_n$, which we rewrite as $[v = (v_1, v_2, \dots, v_{n}), v' = (v_1', v_2', \dots, v_n')]$. By removing $v_1$ and $v_1'$, we have $(v_2, \dots, v_n)$ and $(v_2', \dots, v_n')$ (or $TR$ and $M'R'$, respectively, as Dyck paths). Since $[v, v']$ are intervals in $\mathrm{KTam}_n$, $[(v_2, \dots, v_n), (v_2', \dots, v_n')]$ are intervals in $\mathrm{KTam}_{n-1}$. Conversely, consider the intervals $[(v_1, v_2, \dots, v_{n-1}), (v_1', v_2', \dots, v_{n-1})]$ in $\mathrm{KTam}_{n-1}$. Observe that $[(0, v_1, v_2, \dots, v_{n-1}), (v_1' + 1, v_1', v_2', \dots, v_{n-1}')]$ are intervals in $\mathrm{KTam}_{n}$. By \Cref{lem:ktam_int}, these intervals are of the form $[udTR, uM'dR']$. So, there are $|\mathrm{KTam}_{n-1}|$ intervals $[udTR, uM'dR']$. We have the sum $|\mathrm{KTam}_n| = |\mathrm{KTam}_{n-1}| + \sum_{i=0}^{n-1} |\mathrm{KTam}_i||\mathrm{KTam}_{n-i-1}|$ as desired.
\end{proof}

\subsection{Mapping intervals in the skeletal posets}
\label{subsec:map_int}

At the end of \Cref{subsec:prop_w_sk}, we call the Dyck-Tamari poset $\mathrm{DTam}_n$ the long skeletal poset of the Tamari lattice, and the Kreweras-Tamari poset $\mathrm{KTam}_n$ the wide skeletal poset of the Tamari lattice. This naming convention is meant to highlight their connections. For the remainder of \Cref{sec:sk_tam}, we give further details to what these connections are; first, with a very intriguing observation, then, with the generalizations of these posets for the $m$-Tamari lattice, a generalization of the Tamari lattice using $m$-Dyck paths introduced by Bergeron and Préville-Ratelle~\cite{BPR}.

Recall that cover relations in $\mathrm{DTam}_n$ and in $\mathrm{KTam}_n$ are equinumerous, which follows from the bijection in \Cref{prop:cov_bij}. Naturally, the next step is to extend the map in \Cref{prop:cov_bij} from cover relations to intervals.

\begin{figure}[htp]
    \centering
    \captionsetup{width = 0.75\linewidth}
    \includegraphics[width=0.55\linewidth]{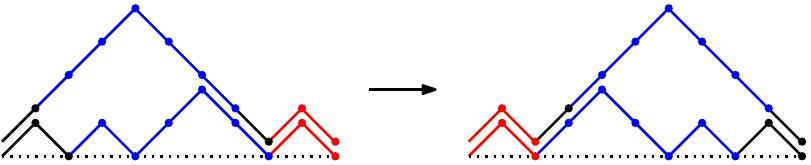}
    \caption{The interval $[(0, \textcolor{blue}{0}, \textcolor{blue}{1}, \textcolor{blue}{0}, \textcolor{red}{0}), (3, \textcolor{blue}{2}, \textcolor{blue}{1}, \textcolor{blue}{0}, \textcolor{red}{0})]$ in $\mathrm{KTam}_5$ mapped to the interval $[(\textcolor{red}{0}, \textcolor{blue}{1}, \textcolor{blue}{0}, \textcolor{blue}{0}, 0), (\textcolor{red}{0}, 3, \textcolor{blue}{2}, \textcolor{blue}{1}, \textcolor{blue}{0}, \textcolor{red}{0})]$ in $\mathrm{DTam}_5$.}
    \label{fig:mirror}
\end{figure}

Consider $AudDB$ in the cover relation $AudDB \lessdot AuDdB$ in $\mathrm{KTam}_n$ by \Cref{def:ktam}. Let $\mathrm{M}$ be the \emph{mirror} map defined as follows: for letters $u$ and $d$ in a Dyck path $w$, $\mathrm{M}: u_k \rightarrow d_{2n-k+1}$ and $\mathrm{M}: d_k \rightarrow u_{2n-k+1}$ for $1 \leq k \leq 2n$. For example, $\mathrm{M}: u_1d_2u_3u_4d_5d_6 \rightarrow u_1u_2d_3d_4u_5d_6$. The idea is to transform $udu$ in $AudDB = AuduB'$, where $uB' = DB$, to $dud$ in \reflectbox{$B'$}$dud$\reflectbox{$A$}. By \Cref{def:dtam}, we have a cover relation \reflectbox{$B'$}$dud$\reflectbox{$A$} $\rightarrow$ \reflectbox{$B'$}$udd$\reflectbox{$A$} in $\mathrm{DTam}_n$. Note that for a cover relation $AudDB \lessdot AuDdB$ in $\mathrm{KTam}_n$, the mirror map $\mathrm{M}$ takes such a cover relation to cover relation $AdudB \lessdot AuddB$ in $\mathrm{DTam}_n$ if and only if $D = ud$. Otherwise, we have the following.

\begin{proposition}[A map from intervals in the Kreweras-Tamari poset to intervals in the Dyck-Tamari poset]
\label{prop:map_int_sk}
    Let $\mathrm{M}$ be a map such that for any letter $u$ and $d$ in a Dyck path, or tuples of Dyck paths, $\mathrm{M}: u_k \rightarrow d_{2n-k+1}$ and $\mathrm{M}: d_k \rightarrow u_{2n-k+1}$ for $1 \leq k \leq 2n$. Consider an interval $[w, w']_{\mathrm{KTam}_n}$ in $\mathrm{KTam}_n$; then $\mathrm{M}: [w, w']_{\mathrm{KTam}_n}$$\rightarrow$$[$\reflectbox{$w$}$,$\reflectbox{$w'$}$]_{\mathrm{DTam}_n}$, where $[$\reflectbox{$w$}$,$\reflectbox{$w'$}$]_{\mathrm{DTam}_n}$ is an interval in $\mathrm{DTam}_n$.
\end{proposition}

\begin{proof}
    It suffices to consider the case where $[w, w']$ is a cover relation in $\mathrm{KTam}_n$. The singleton case, that is $w = w'$, holds by reflexivity. Consider a cover relation $AudDB \lessdot AuDdB$ in $\mathrm{KTam}_n$. Observe that $\mathrm{M}(AudDB) = \mathrm{M}(B)\mathrm{M}(D)ud\mathrm{M}(A)$ and $\mathrm{M}(AuDdB) = \mathrm{M}(B)u\mathrm{M}(D)d\mathrm{M}(A)$. We may ignore $\mathrm{M}(A)$ and $\mathrm{M}(B)$. So, it remains to show that $\mathrm{M}(D)ud \leq u\mathrm{M}(D)d$. By \Cref{cor:comp}, $\mathrm{M}(D)ud \leq u\mathrm{M}(D)d$ if and only if $\mathrm{M}(D)ud \neq u \dots duu \dots d$. Since $\mathrm{M}(duu) = ddu$ and cover relations in $\mathrm{KTam}_n$ are of the form $AudDB \lessdot AuDdB$, $\mathrm{M}(D)ud \neq u \dots duu \dots d$. We conclude that $\mathrm{M}$ maps cover relations in $\mathrm{KTam}_n$ to intervals in $\mathrm{DTam}_n$. The general case follows by transitivity. We have the proposition as claimed.  
\end{proof}

The converse of \Cref{prop:map_int_sk} does not hold by \Cref{thm:count_dtam} and \Cref{thm:count_ktam}. For this preliminary version, we do not characterize intervals in $\mathrm{DTam}_n$ that are in bijection with intervals in $\mathrm{KTam}_n$. 

\subsection{The m-skeletal posets of the m-Tamari lattice}
\label{subsec:m_sk}

We end \Cref{sec:sk_tam} with definitions of the $m$-Dyck-Tamari and $m$-Kreweras-Tamari posets, i.e., the skeletal posets of the $m$-Tamari lattice, and recount discussions relevant for future work. To define the $m$-Dyck-Tamari poset $m$-$\mathrm{DTam}_n$ and the $m$-Kreweras-Tamari poset $m$-$\mathrm{KTam}_n$, we need $m$-Dyck paths; a generalization of Dyck paths where every letter $u$ belongs to a sequence of consecutive letters $u$.

\begin{definition}[$m$-Dyck paths, folklore]
\label{def:m_Dyck}
    A $m$-Dyck path $w = w_1w_2 \cdots w_{(m+1)n}$ for $m \in \mathbb{N}$ is a Dyck path of size $mn$ where every letter $u$ in $w$ belongs to a sequence of $mk$ consecutive letters $u$ for $1 \leq k \leq n$.
\end{definition}

\begin{figure}[htp]
    \centering
    \includegraphics[width=0.35\linewidth]{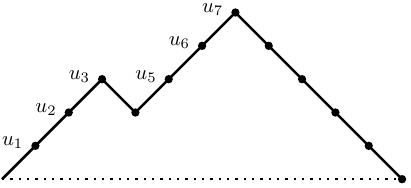}
    \caption{An example of a $3$-Dyck path.}
    \label{fig:3_Dyck}
\end{figure}

Since the $m$-Tamari lattice is an interval in the Tamari lattice restricted to $m$-Dyck paths, we do not need to (re)define it here. However, the same can not be said for the skeletal posets of the $m$-Tamari lattice. To define these posets, we extend the bijection from \Cref{prop:cov_bij}. We recommend that the reader see \Cref{fig:cov_bij} before continuing. 

For a Dyck path $w$, the bijection from \Cref{prop:cov_bij} involves adding $ud$ after a letter $d$ in $w$ or before its paired letter $u$ in $w$. By construction, the resulting path naturally maps to another via the cover relation $AdudB \lessdot AuddB$ in $\mathrm{DTam}_n$ by \Cref{def:dtam} or the cover relation $AudDB \lessdot AuDdB$ in $\mathrm{KTam}_n$ by \Cref{def:ktam}. 

Using this idea of adding $ud$, we extend our bijection to $m$-Dyck paths. Consider the figure below. 

\begin{figure}[htp]
    \centering
    \includegraphics[width=0.75\linewidth]{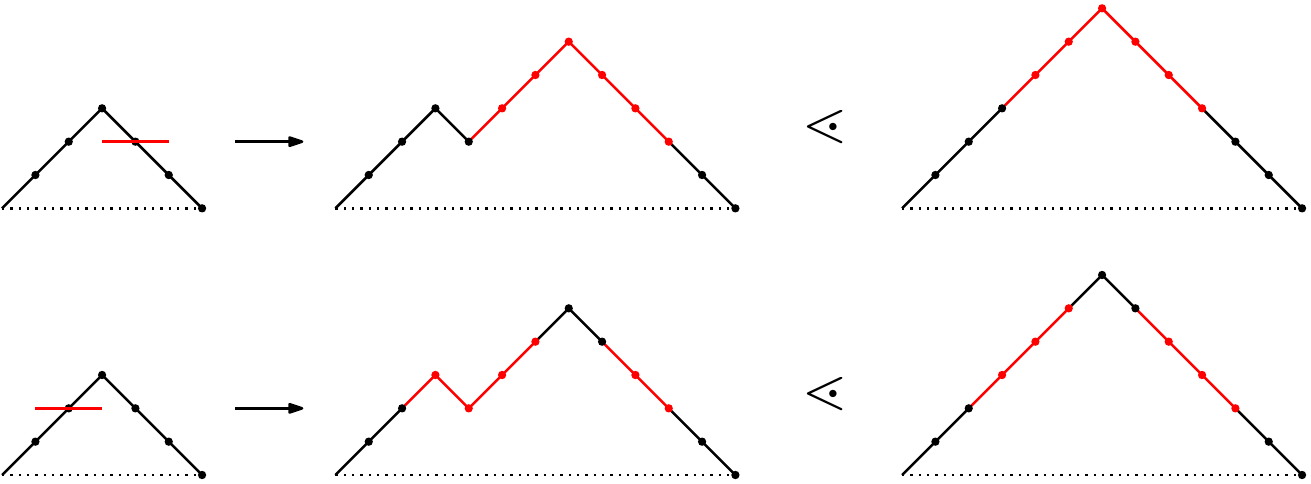}
    \caption{The bijection from \Cref{prop:cov_bij} for $3$-Dyck paths. \textbf{Top:} from a $3$-Dyck path to a cover relation in $3$-$\mathrm{DTam}_n$. \textbf{Bottom:} from a $3$-Dyck path to a cover relation in $3$-$\mathrm{KTam}_n$. In this case as depicted, they are the same.} 
    \label{fig:3_cover}
\end{figure}

For a $m$-Dyck path $w$, we add $u^md^m$ after a letter $d$ as shown in the top half of \Cref{fig:3_cover}. The desired cover relation is $Adu^md^mB \lessdot Au^md^{m+1}B$, which defines the $m$-Dyck-Tamari poset $m$-$\mathrm{DTam}_n$. Consider the same letter $d$ and its paired letter $u$. Unlike the case $m=1$, adding $u^md^m$ before a letter $u$ may map the $m$-Dyck path to an $m'$-Dyck path where $m \neq m'$. So, we proceed as follows:

\begin{itemize}
    \item we write $w$ as $Au^ku^{m-k}B$ for $1 \leq k \leq m$, and split $u^md^m$ into $u^{m-k}d^{m-k}$ and $u^kd^k$,
    \item we add $u^{m-k}d^{m-k}$ after $Au^k$ to get $Au^ku^{m-k}d^{m-k}u^{m-k}B$,
    \item $u^{m-k}$ is the first consecutive letters $u$ in the Dyck path $D$ of minimal size following $d^{m-k}$, thus $Au^ku^{m-k}d^{m-k}u^{m-k}B = Au^ku^{m-k}d^{m-k}DB'$, where $u^{m-k}B = DB'$,
    \item we add $u^k$ before $D$ and $d^k$ after $D$ to get $Au^ku^{m-k}d^{m-k}u^kDd^kB$, and
    \item rewrite $Au^k$ as $A'$, so, 
    \item the desired cover relation is $A'u^{m-k}d^{m-k}u^kDd^kB' \lessdot A'u^{m-k}d^{m-k-1}u^{k}Dd^{k+1}B'$, which \newline defines the $m$-Kreweras-Tamari poset $m$-$\mathrm{KTam}_n$.
\end{itemize}

By construction, $u^kDd^k$ is a $m$-Dyck path. So, $A'u^{m-k}d^{m-k}u^kDd^kB' \lessdot A'u^{m-k}d^{m-k-1}u^{k}Dd^{k+1}B'$ can be written as $Au^kd^kDB \lessdot Au^kd^{k-1}DdB$ for $1 \leq k \leq m$, and $D$ is the $m$-Dyck path of minimal size. We have $mnC_n^m$ cover relations $Adu^md^mB \lessdot Au^md^{m+1}B$ and $mnC_n^{m}$ cover relations $Au^kd^kDB \lessdot Au^kd^{k-1}DdB$, where $C_n^m = \frac{1}{mn+1}\binom{(m+1)n}{n}$ is a Fuss-Catalan number. The reader may check that the case $m=1$ matches \Cref{def:dtam}, \Cref{def:ktam}, and \Cref{prop:cov_bij}. We define, formally, the $m$-Dyck-Tamari poset $m$-$\mathrm{DTam}_n$ and the $m$-Kreweras-Tamari poset $m$-$\mathrm{KTam}_n$ below.

\begin{definition}[The $m$-Dyck-Tamari poset]
\label{def:m_dtam}
    The \emph{m-Dyck-Tamari poset} $m$-$\mathrm{DTam}_n$ for $m \in \mathbb{N}$ is the reflexive transitive closure of cover relations $w = Adu^md^mB \lessdot w' = Au^md^{m+1}B$, where $w$ and $w'$ are $m$-Dyck paths of size $mn$ and $A$ is nonempty.
\end{definition}

\begin{definition}[The $m$-Kreweras-Tamari poset]
\label{def:m_ktam}
    The \emph{m-Kreweras-Tamari poset} $m$-$\mathrm{KTam}_n$ for $m \in \mathbb{N}$ is the reflexive transitive closure of cover relations $Au^kd^kDB \lessdot Au^kd^{k-1}DdB$ for $1 \leq k \leq m$, where $w$ and $w'$ are $m$-Dyck paths of size $mn$, and $D$ is the nonempty $m$-Dyck path of minimal size.
\end{definition}

\subsection{Skeletal posets for Cambrian lattices?}
\label{subsec:camb_sk}

When we found the Dyck-Tamari poset $\mathrm{DTam}_n$ and the Kreweras-Tamari poset $\mathrm{KTam}_n$, we did not expect their intervals to be so well behaved and closely connected; as seen with the bijection between their cover relations and the map from intervals in $\mathrm{KTam}_n$ to intervals in $\mathrm{DTam}_n$. So, we wondered whether these posets, collectively called skeletal posets, can be generalized to Cambrian lattices. 

Cambrian lattices were introduced by Reading~\cite{Rea1} as a generalization of the Tamari lattice in the Coxeter groups direction; notably to Coxeter group of other types besides $A$ (the Tamari lattice is a Cambrian lattice of type $A$). Cambrian lattices can be defined on words. In this setting, Chenevière suggests a type uniform definition for skeletal posets of Cambrian lattices. For cover relations in a Cambrian lattice, one may label them by the differences in lengths or absolute lengths. Chenevière suggested that the skeletal posets be defined by the restriction of the differences in lengths and absolute lengths to $1$. At the time of this writing, we have not explored this direction. Alternatively, since Cambrian lattices (of type $A$) may be defined using Cambrian trees of Châtel and Pilaud~\cite{CP}, we may then define skeletal posets via restrictions on tree rotations. This is our approach in~\cite{L1}.

Besides studying these posets by themselves, they may be used to find other lattices; see \Cref{sec:alt} for an approach to doing so with $\mathrm{DTam}_n$, though this is not the purpose of \Cref{sec:alt}.  

\section{Altitude lattices}
\label{sec:alt}

This section is independent of \Cref{sec:sk_tam} and \Cref{sec:kneser}. The goal here is to introduce altitude lattices, which are inspired by the altitude-Tamari lattices, or alt-Tamari for short, of Chenevière. By construction, altitude lattices within the same family refine a well-chosen distributive lattice. They are anatomically related and have the same number of linear intervals. When this distributive lattice is the Dyck lattice, the family of altitude lattices contains the alt-Tamari lattices among other lattices. Unfortunately, due to our approach, we do not know whether all of these lattices are already in the literature. For example, for $n=4$, there are $4$ alt-Tamari lattices (see section $4.1$ and remark $4.3$ in~\cite{C1}). With our approach, we found $4$ altitude lattices in this family that are not the Dyck and Tamari lattices (see \Cref{fig:samples}).

\subsection{An order-theoretic approach to linear intervals}
\label{subsec:study_lin_int}

The alt-Tamari lattices of Chenevière, to the best of our knowledge, initiated the study of families of lattices having the same number of linear intervals. The alt $\nu$-Tamari lattices of Ceballos and Chenevière~\cite{CC} and, more generally, some framing lattices of von Bell and Ceballos~\cite{vBC}, are examples of such families. They inspired us to study other posets with the same number of linear intervals. 

For this paper, we focus on posets with some intervals isomorphic to the Dyck lattice $\mathrm{Dyck}_3$ or the Tamari lattice $\mathrm{Tam}_3$. Specifically, we study such posets where one refines (resp., extends) another. We give a local approach to refining (resp., extending) these posets that preserves the number of linear intervals. With our approach, we show that there exist lattices (besides the alt-Tamari lattices and the Cambrian lattices) with the same number of linear intervals as the Dyck and Tamari lattices. By construction, they refine the Dyck lattice. Finally, all of the lattices in this section are posets induced by embeddings of skeletal posets into the grid, which lead naturally to some extremal problems.

\subsection{On linear intervals}
\label{subsec:lin_int}

We begin with the figure below and some brief remarks for this section. 

\begin{figure}[htp]
    \centering
    \includegraphics[width=0.4\linewidth]{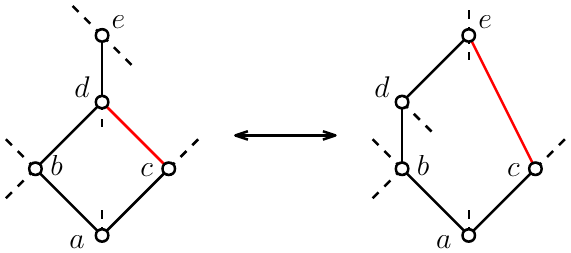}
    \caption{The Dyck lattice $\mathrm{Dyck}_3$ (left) and the Tamari lattice $\mathrm{Tam}_3$ (right) as intervals in some posets $\mathcal{P}$ and $\mathcal{Q}$, respectively. $- -$ indicate the existence of cover relations involving elements not in these intervals. $\longleftrightarrow$ indicates the maps between $\mathcal{P}$ and $\mathcal{Q}$, i.e., refinement of $\mathcal{P}$/extension of $\mathcal{Q}$, involving $c \textcolor{red}{\lessdot} d$ and $c \textcolor{red}{\lessdot} e$ in \textcolor{red}{red}. In this figure, $\mathcal{P}$ and $\mathcal{Q}$ have the same number of linear intervals (see \Cref{prop:same_lin_1}).}
    \label{fig:refine}
\end{figure}

Moving forward, we let $\mathcal{P}$ be a poset with $\mathrm{Dyck}_3$ as an interval, and $\mathcal{Q}$ be a poset obtained by refining that interval to get $\mathrm{Tam}_3$, and vice versa. For convenience, $a, b, c, d,$ and $e$ are always labeled elements in these intervals. We write $\dots$$\lessdot a \lessdot b \lessdot d \lessdot \dots$ to mean linear intervals that contain $[a, d]$; similarly, $\dots$$\lessdot c \lessdot d \lessdot \dots$, $\dots$$\lessdot c \lessdot e \lessdot \dots$, etc. Since the dual of a poset does not change its number of linear intervals, i.e., changing $\lessdot$ to $\gtrdot$ in $\mathcal{P}$ and in $\mathcal{Q}$, we only consider the case $\lessdot$. 

Our goal is to study how the number of linear intervals in $\mathcal{P}$ and $\mathcal{Q}$ changes when one is mapped to the other. Later in the section, we will iteratively refine $\mathcal{P}$ to get $\mathcal{Q}$. For now, we focus on different cases of $\mathcal{P}$ and $\mathcal{Q}$. Let us study the Dyck and Tamari lattices in \Cref{fig:refine} without the cover relations $- -$. The reader may check that these lattices have the same number of linear intervals: $5$ elements, $5$ cover relations, and $2$ linear intervals that are $3$-element chains. On the one hand, in the Dyck lattice, $c \lessdot d$ is a linear interval but not $[a, d]$. On the other hand, in the Tamari lattice, $a \lessdot b \lessdot d$ is a linear interval, but the elements $c$ and $d$ are incomparable. If $a \lessdot b \lessdot d$ and $c \lessdot d$ are the only linear intervals that contain elements $a, b, d$ and elements $c, d$, respectively, then $\mathcal{P}$ and $\mathcal{Q}$ have the same number of linear intervals. 

For a more general case, consider $\mathcal{P}$ and $\mathcal{Q}$ as depicted in \Cref{fig:refine}. Note that any linear interval not involving $a \lessdot b \lessdot d$, $c \lessdot d$, or $c \lessdot e$ are in $\mathcal{P}$ and in $\mathcal{Q}$. We show that $\mathcal{P}$ and $\mathcal{Q}$ have the same number of linear intervals. 

\begin{proposition}[Posets with the same number of linear intervals, I]
\label{prop:same_lin_1}
    Suppose $a$ is a minimal element and there does not exist elements $x \neq a$ and $y \neq e$ such that $x \lessdot c$ and $d \lessdot y$. Then $\mathcal{P}$ and $\mathcal{Q}$ have the same number of linear intervals. 
\end{proposition}

\begin{proof}
    It suffices to consider linear intervals other than the chain $a \lessdot b \lessdot d$ and the cover relation $c \lessdot d$. Since $a$ is a minimal element and there does not exist an element $y$ such that $d \lessdot y$, the only linear interval with elements $a, b$ and $d$ is $a \lessdot b \lessdot d$ in $\mathcal{Q}$. 
    
    Otherwise, consider linear intervals that contain the cover relation $c \lessdot d$. There does not exist an element $x$ such that $x \lessdot c$. So, we have the linear intervals $c \lessdot d \lessdot e \lessdot \dots$ in $\mathcal{P}$, which are of the form $c \lessdot e \lessdot \dots$ in $\mathcal{Q}$, and vice versa. Since linear intervals $d \lessdot e \lessdot \dots$ are in $\mathcal{P}$ and in $\mathcal{Q}$, the numbers of linear intervals $c \lessdot d \lessdot e \lessdot \dots$ in $\mathcal{P}$ and linear intervals $c \lessdot e \lessdot \dots$ in $\mathcal{Q}$ are the same. It remains to check that there are no other linear intervals with $[c, e]$. 

    \begin{figure}[htp]
        \centering
        \includegraphics[width=0.35\linewidth]{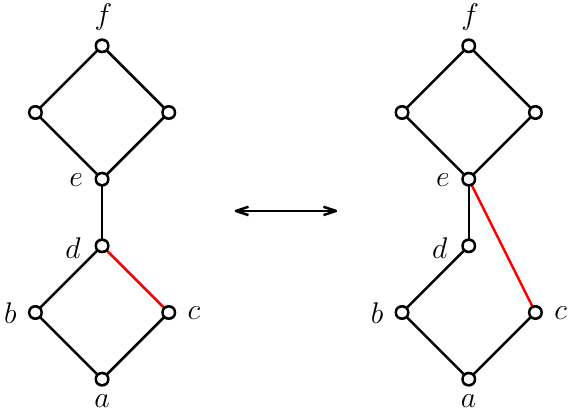}
    \end{figure}
    
    Consider an interval $[c, f]$. If $[c, f]$ is not linear in $\mathcal{P}$, then there exist at least two elements in $[c, f]$ that are incomparable as shown in the figure above. Observe that all elements bigger than $d$ are bigger than or equal to $e$. We have $c \lessdot d \lessdot e \leq f$ in $\mathcal{P}$ and $c \lessdot e \leq f$ in $\mathcal{Q}$. So, $[c, f]$ is not linear in $\mathcal{Q}$, and the converse holds. It follows that linear intervals with the cover relation $c \lessdot e$ are exactly $c \lessdot e \lessdot \dots$ in $\mathcal{Q}$. Since the number of such linear intervals is the same as the number of linear intervals $c \lessdot d \lessdot e \lessdot \dots$ in $\mathcal{P}$, $\mathcal{P}$ and $\mathcal{Q}$ have the same number of linear intervals. The proposition holds as claimed. 
\end{proof}

For $\mathcal{P}$ and $\mathcal{Q}$ to have the same number of linear intervals, it suffices that the number of linear intervals that are only in $\mathcal{P}$ is the same as the number of linear intervals that are only in $\mathcal{Q}$. So, we look for linear intervals that are in $\mathcal{P}$ or $\mathcal{Q}$ but not both. By observation, we have linear intervals $\dots$$\lessdot c \lessdot d \lessdot y \lessdot \dots$ in $\mathcal{P}$ but not in $\mathcal{Q}$, and linear intervals $\dots$$\lessdot a \lessdot b \lessdot d \lessdot \dots$ in $\mathcal{Q}$ but not in $\mathcal{P}$. Note that linear intervals $\dots$$c \lessdot d \lessdot y \lessdot \dots$ are exactly those in $\mathcal{P}$ but not in $\mathcal{Q}$. Indeed, since $c \lessdot d$ is in $\mathcal{P}$ and not in $\mathcal{Q}$, linear intervals without $c \lessdot d$ are either in both $\mathcal{P}$ and $\mathcal{Q}$ or only in $\mathcal{Q}$. Thus, we need only find linear intervals that are in $\mathcal{Q}$ but not in $\mathcal{P}$.

\begin{figure}[htp]
    \centering
    \includegraphics[width=0.8\linewidth]{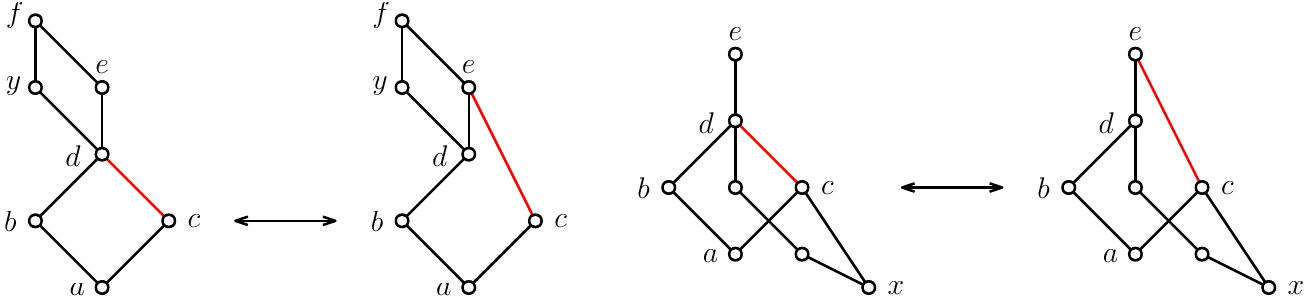}
    \caption{Two examples of $\mathcal{P}$ and $\mathcal{Q}$; each with an interval that is nonlinear in $\mathcal{P}$ but linear in $\mathcal{Q}$.}
    \label{fig:ce}
\end{figure}

Consider an interval that is nonlinear in $\mathcal{P}$ but linear in $\mathcal{Q}$. Such an interval has exactly two maximal chains in $\mathcal{P}$ and at least one contains the cover relation $c \lessdot d$. Furthermore, it has exactly one maximal chain in $\mathcal{Q}$ since it is linear. For $\mathcal{P}$ and $\mathcal{Q}$ in \Cref{fig:ce}, observe that intervals $[x, d]$ and $[c, f]$ are nonlinear in $\mathcal{P}$ but linear in $\mathcal{Q}$. In both instances, a maximal chain containing $c \lessdot d$ is removed or shortened, e.g., $c \lessdot d \lessdot e \lessdot f$ to $c \lessdot e \lessdot f$. It follows that we have linear intervals $\dots$$\lessdot x \lessdot \dots \lessdot d \lessdot \dots$ and $\dots$$\lessdot c \lessdot e \lessdot f \lessdot \dots$, where $e \lessdot f$ in $\mathcal{Q}$ but not in $\mathcal{P}$. Since there are no other linear intervals in $\mathcal{Q}$ but not in $\mathcal{P}$, we now have all of the linear intervals that are either in $\mathcal{P}$ or in $\mathcal{Q}$. As a result, we obtain the following lemma, which is a more general version of \Cref{prop:same_lin_1}.

\begin{lemma}[Posets with the same number of linear intervals, II]
\label{lem:same_lin}
    Let $\mathcal{P}$ be a poset with an interval isomorphic to $\mathrm{Dyck}_3$, and $\mathcal{Q}$ the refinement of $\mathcal{P}$ obtained by refining the interval $\mathrm{Dyck}_3$ in $\mathcal{P}$ to $\mathrm{Tam}_3$. Consider elements $a, b, c, d$ and $e$ in these intervals as labeled in \Cref{fig:refine}. Let
    \begin{itemize}
        \item $p$ be the number of linear intervals $\dots$$\lessdot c \lessdot d \lessdot y \lessdot \dots$ in $\mathcal{P}$, where $y \neq e$ and $d \lessdot y$, or $\dots$$\lessdot c \lessdot d$ when there does not exist a $y$,
        \item $q_1$ be the number of linear intervals $\dots$$\lessdot a \lessdot b \lessdot d \lessdot \dots$ in $\mathcal{Q}$, 
        \item $q_2$ be the number of linear intervals $\dots$$\lessdot x \lessdot \dots \lessdot d \lessdot \dots$ in $\mathcal{Q}$, where $x \neq a$ and $x \lessdot c$, and
        \item $q_3$ be the number of linear intervals $\dots$$\lessdot c \lessdot e \lessdot f \lessdot \dots$ in $\mathcal{Q}$, where $e \lessdot f$.
    \end{itemize}
    $\mathcal{P}$ and $\mathcal{Q}$ have the same number of linear intervals if and only if $p = q_1 + q_2 + q_3$ where $p, q_1 \geq 1$ and $q_2, q_3 \geq 0$. 
\end{lemma}

\begin{remark}[On other intervals]
    The reasoning in \Cref{prop:same_lin_1} and \Cref{lem:same_lin} can be used for other intervals besides the Dyck and Tamari intervals (see the figure below), which we do not consider in this paper. We omit the double-sided arrows moving forward.
    \begin{figure}[htp]
        \centering
        \includegraphics[width=0.65\linewidth]{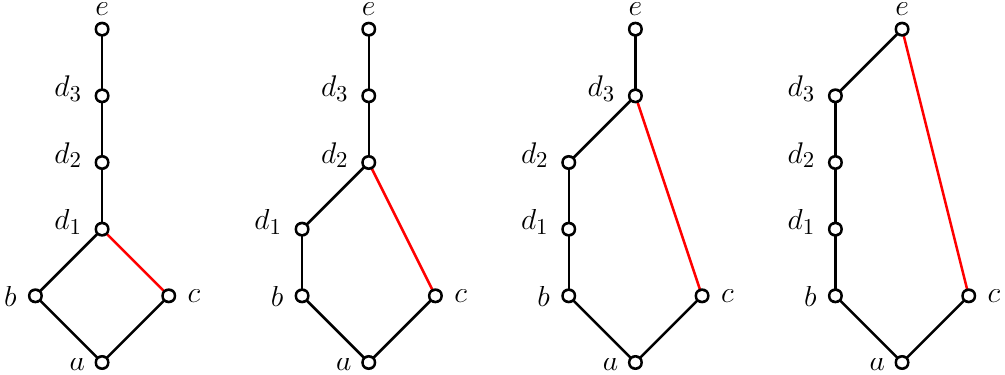}
        \caption{A small example of lattices with the same number of linear intervals.}
    \end{figure}
\end{remark}

Our goal with \Cref{lem:same_lin} is to give a universal approach to extending or refining any poset $\mathcal{P}$ or $\mathcal{Q}$. However, in practice, it may be overly tedious to check that $p = q_1 + q_2 + q_3$. In an effort to convince the reader that our approach can be useful and interesting (for some posets $\mathcal{P}$ and $\mathcal{Q}$), we use \Cref{lem:same_lin} to find lattices with the same number of linear intervals as the Dyck lattice. In~\cite{L2}, we study ways to circumvent some of the counting in \Cref{lem:same_lin}. For this paper, we briefly describe one of such ways via lattices embedded into the grid. Besides further directions, the rest of \Cref{sec:alt} is dedicated to explaining how we obtain the lattices in \Cref{fig:samples} starting from $\mathrm{Dyck}_4$. We use these lattices to motivate the definition of altitude lattices, a generalization of the alt-Tamari lattices for well-chosen distributive lattices, even in the case of the Dyck lattice. These lattices are the primary examples where our approach is most fruitful. Notably, they have as many Dyck and Tamari intervals as possible. 

\begin{figure}[htp]
    \centering
    \includegraphics[width=0.65\linewidth]{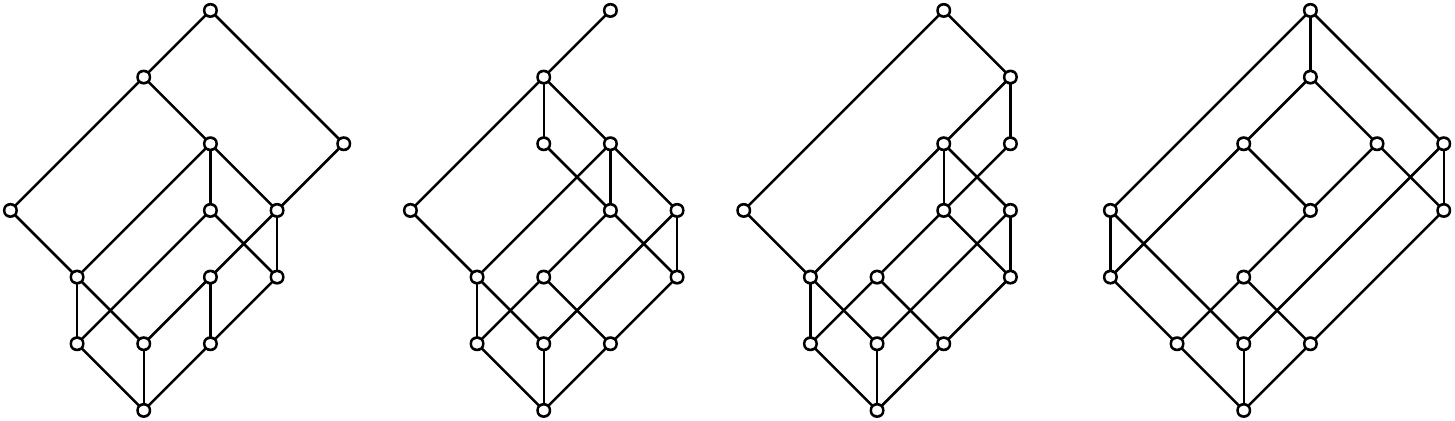}
    \caption{The Hasse diagrams of some lattices with the same number of linear intervals as the Dyck lattice (for $n=4$).}
    \label{fig:samples}
\end{figure}

\subsection{Polygonal lattices in the grid}
\label{subsec:alt}

Formally, the lattices in \Cref{fig:samples} are \textit{polygonal} lattices in the sense of Reading~\cite{Rea2}. Intuitively and visually, there are as many polygons as possible (see the definition below). A familiar example is the Boolean lattice, the $1$-skeleton of the hypercube, where every polygon is a square. 

\begin{definition}[Polygonal lattice, Definition 9-6.1 in~\cite{Rea2}]
    A \emph{polygon} in a lattice $\mathcal{L}$ is an interval $[x, y]$ that is the union of two ﬁnite maximal chains from $x$ to $y$, with these chains disjoint except at $x$ and $y$. A lattice $\mathcal{L}$ is called \emph{polygonal} if the following two dual conditions hold:
    \begin{itemize}
        \item[(i)] If distinct elements $y_1$ and $y_2$ cover an element $x$, then $[x, y_1 \vee y_2]$ is a polygon.
        \item[(ii)] If an element $y$ covers distinct elements $x_1$ and $x_2$, then $[x_1 \wedge x_2, y]$ is a polygon.
    \end{itemize}
\end{definition}

Since the Dyck and Tamari intervals have a square and a pentagon, respectively, we want polygonal lattices with as many squares and pentagons as possible. As such, the most natural starting point is to take a distributive lattice where every polygon is necessarily a square. 

By a classic theorem of Dilworth~\cite{D}, we can embed distributive lattices into the grid. Furthermore, such an embedding is isometric, i.e., where every cover relation corresponds to a line segment of unit length. We use the grid as guide to go from $\mathrm{Dyck}_4$ to $\mathrm{Tam}_4$ as shown in \Cref{fig:tam}. For now, consider the figure below.

\begin{figure}[htp]
    \centering
    \includegraphics[width=0.5\linewidth]{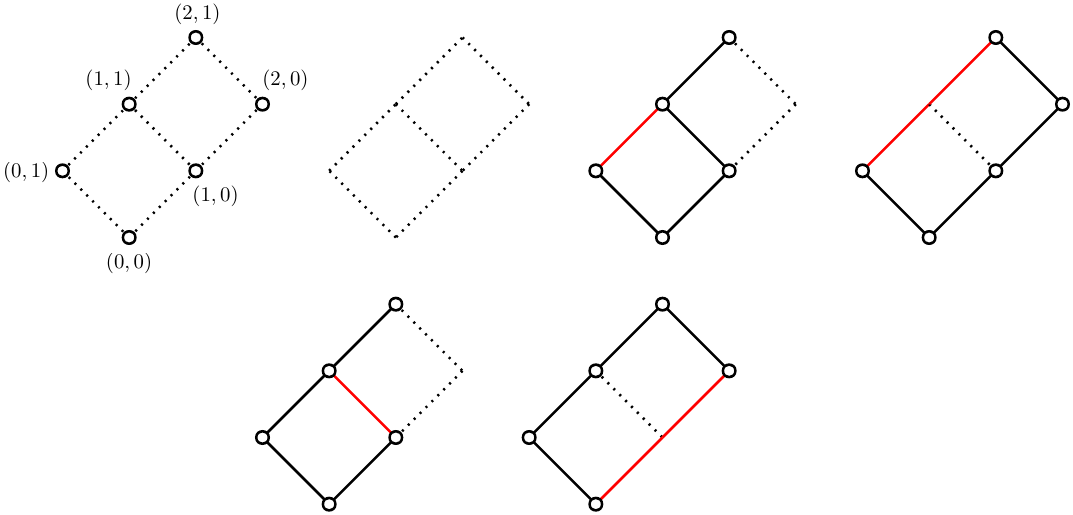}
    \caption{Smallest examples of $\mathcal{P}$ and $\mathcal{Q}$ embedded in the grid $\mathbb{Z}_{\geq 0}^{2}$.}
    \label{fig:grid}
\end{figure} 

Geometrically, observe that a cover relation corresponds to a line segment along an axis. Furthermore, the height $h(x_\downarrow)$ for any element $x$ is invariant under the refinements and extensions of $\mathcal{P}$ and $\mathcal{Q}$, respectively. In certain cases, $h(x_\downarrow)$ is the sum of coordinates of $x$ in the grid.

We mentioned earlier that they are ways of refining $\mathcal{P}$ or extending $\mathcal{Q}$ that circumvent some of the counting in \Cref{lem:same_lin}. One that is particularly worth highlighting is when $\mathcal{P}$ and $\mathcal{Q}$ have intervals that are products of Dyck and Tamari intervals, respectively. Consider the figure below.

\begin{figure}[htp]
    \centering
    \includegraphics[width=0.3\linewidth]{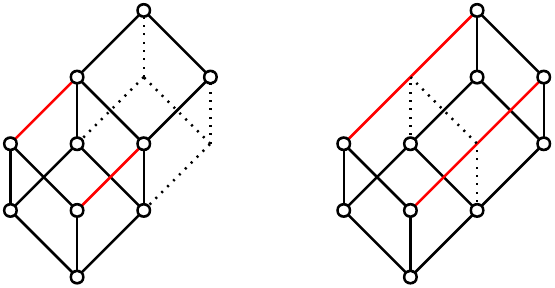}
    \caption{An extension/refinement forcing another in the grid embedding.}
    \label{fig:grid_move}
\end{figure}

By construction, we want $\mathcal{P}$ and $\mathcal{Q}$ to always be embeddable into the grid as described earlier with \Cref{fig:grid}. If $\mathcal{P}$ and $\mathcal{Q}$ are products of Dyck or Tamari intervals, respectively, then the extension/refinement of an interval forces the same extension/refinement for each interval in the product. Note that $\mathcal{P}$ and $\mathcal{Q}$ have the same number of linear intervals. For the sake of discussion, let us ignore the grid embedding and allow the extension/refinement of an interval in a product. 

\begin{figure}[htp]
    \centering
    \includegraphics[width=0.3\linewidth]{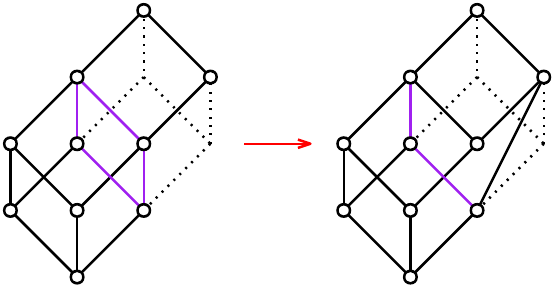}
    \caption{$\mathcal{P}$ has less linear intervals than $\mathcal{Q}$.}
    \label{fig:grid_counter}
\end{figure}

Consider the figure above. Observe that the purple interval is nonlinear in $\mathcal{P}$ but is linear in $\mathcal{Q}$. $\mathcal{P}$ has $4$ linear intervals that are $3$ element chains while $\mathcal{Q}$ has $5$ of such linear intervals. This is an example where the grid allow us to ignore some undesirable cases as they are not realizable by construction. In general, it allow us to ignore linear intervals in products of Dyck and Tamari intervals and introduces the idea of using multiple refinements/extensions. As a final note before the examples, there are many ways to embed a poset into the grid. For example, we use the following two embeddings of the Dyck lattice, which we take for granted in this paper.     

\begin{figure}[htp]
    \centering
    \includegraphics[width=0.45\linewidth]{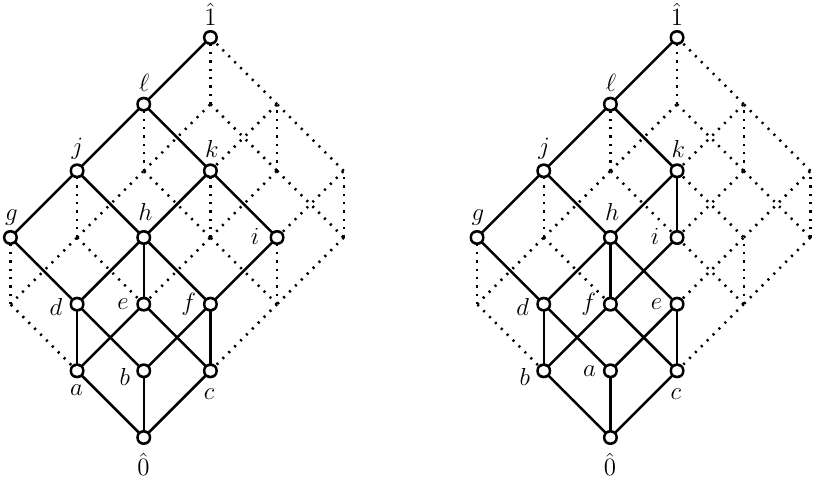}
    \caption{Two embeddings of the Dyck lattice in $\mathbb{Z}_{\geq 0}^{3}$.}
    \label{fig:drawing_dyck}
\end{figure}

\subsection{Extensions and refinements of the Dyck lattice in the grid}
\label{subsec:grid}

Now that we have the necessary setup, we focus on two examples. In each example, we refine the Dyck lattice to get lattices with the same number of linear intervals. Each sequence of refinements stops whenever it reaches a ``largest'' lattice (we use ``largest'' in the geometric sense, i.e., convex hull). For each lattice, we list their linear intervals that are $3$ or $4$-element chains since refining $\mathcal{P}$ or extending $\mathcal{Q}$ does not change the numbers of elements and cover relations. Then, we briefly explain the reasoning behind each refinement with some maps between linear intervals with $3$ or $4$ elements.

\begin{example}[Refining the Dyck lattice, I]
\label{exm:refine_dyck_1}

For the first example, we show a sequence of refinements that ends with the Tamari lattice.

    \begin{figure}[htp]
        \centering
        \includegraphics[width=0.85\linewidth]{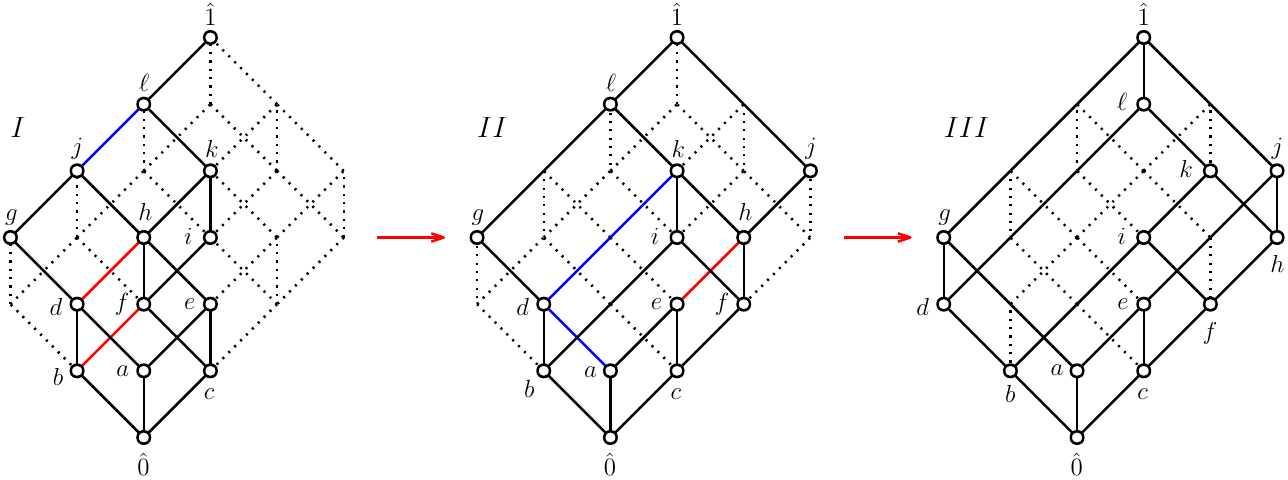}
        \caption{From the Dyck lattice to the Tamari lattice (for $n=4$). $A.I/B.I$ and $A.III/B.IV$ are the Dyck and Tamari lattices, respectively.}        
        \label{fig:dyck_to_tam}
    \end{figure}

Below is the list of $14$ linear intervals for each lattice; $12$ linear intervals are $3$ element chains and $2$ linear intervals are $4$ element chains. We recommend that the reader note the differences in linear intervals.

    \begin{itemize}
        \item[$I$.] The $3$-element linear intervals are $[a, g]$, $[b, i]$, $[b, g]$, $[c, i]$, $[d, k]$, $[e, j]$, $[e, k]$, $[f, j]$, $[g, \ell]$, $[i, \ell]$, $[j, \hat{1}]$, $[k, \hat{1}]$, and the $4$-element linear intervals are $[g, \hat{1}]$ and $[i, \hat{1}]$. 
        \item[$II$.] The $3$-element linear intervals are $[\hat{0}, f]$, $[a, g]$, $[a, h]$, $[b, g]$, $[c, i]$, $[e, j]$, $[e, k]$, $[f, j]$, $[g, \hat{1}]$, $[i, \ell]$, $[h, \ell]$, $[k, \hat{1}]$, and the $4$-element linear intervals are $[a, j]$ and $[i, \hat{1}]$.
        \item[$III$.] The $3$-element linear intervals are $[\hat{0}, d]$, $[\hat{0}, f]$, $[a, j]$, $[b, g]$, $[b, k]$, $[c, h]$, $[c, i]$,$[e, \hat{1}]$, $[f, j]$, $[i, \ell]$, $[h, \ell]$, $[k, \hat{1}]$, and the $4$-element linear intervals are $[\hat{0}, h]$ and $[i, \hat{1}]$.
    \end{itemize}

The reasoning behind each refinement is as follows:

    \begin{itemize}
        \item[$I \rightarrow II$.]
        \begin{itemize}
            \item[(i)] Consider $[\hat{0}, i]$, $[a, k]$, and $[h, \hat{1}]$ in both $A.I$ and $A.II$. Note that $k_\downarrow$ is the product of two Dyck intervals which explains the refinements of $[\hat{0}, i]$ and $[a, k]$. For $[h, \hat{1}]$, see the reasoning below. We have $[b, i]_I \rightarrow [\hat{0}, f]_{II}$, $[d, k]_I \rightarrow [a, h]_{II}$, and $[j, \hat{1}]_I \rightarrow [h, \ell]_{II}$.
            \item[(ii)] The refinement of $[h, \hat{1}]$ is forced by the refinements in $k_\downarrow$, which is why the cover relation $j \lessdot \ell$ is colored blue. The key observation here is that $h$ also moves $j$ with it in the grid. Hence, we also have $[g, \hat{1}]_{I} \rightarrow [a, j]_{II}$.
        \end{itemize}  
        \item[$II \rightarrow III$.]
        \begin{itemize}
            \item[(i)] The refinement of $[c, j]$ forces the refinements of $[\hat{0}, g]$ and $[b, \ell]$ through $d$ and $k,\ell$ respectively. The refinement of $[b, \ell]$ is impossible without the refinement of $[\hat{0}, g]$.   
            \item[(ii)] Due to multiple refinements, the maps for linear intervals are more complex. We list the simple ones first: $[e, j]_{II} \rightarrow [c, h]_{III}$, $[a, h]_{II} \rightarrow [\hat{0}, h]_{III}$, and $[b, g]_{II} \rightarrow [\hat{0}, d]_{III}$.
            \item[(iii)] Consider $[e, k]_{II}, [d, \ell]_{II}$ and $[g, \hat{1}]_{II}$. The refinement of $[c,j]$ results in $e$ and $k$ being incomparable. Thus, we lost the linear interval $[e, k]_{II}$. In return, $[e, \hat{1}]_{III}$ is linear. So, we have $[e, k]_{II} \rightarrow [e, \hat{1}]_{III}$. For the remaining intervals, $[d, \ell]_{II}$ is not linear and we do not have $[d, \ell]_{II} \rightarrow [b, k]_{III}$. However, observe that $[d, \ell]_{II}$ and $[g, \hat{1}]_{II}$ are in the interval $[d, \hat{1}]_{II}$. So, we replace $[d, \ell]_{II}$ with $[g, \hat{1}]_{II}$ to get $[g, \hat{1}]_{II} \rightarrow [b, k]_{III}$.   
        \end{itemize}
    \end{itemize}
\end{example}

\begin{example}[Refining the Dyck lattice, II]
\label{exm:refine_dyck_2}

For the second example, we show a sequence of refinements that ends with a Cambrian lattice.

\begin{figure}[!ht]
        \centering
        \includegraphics[width=1\linewidth]{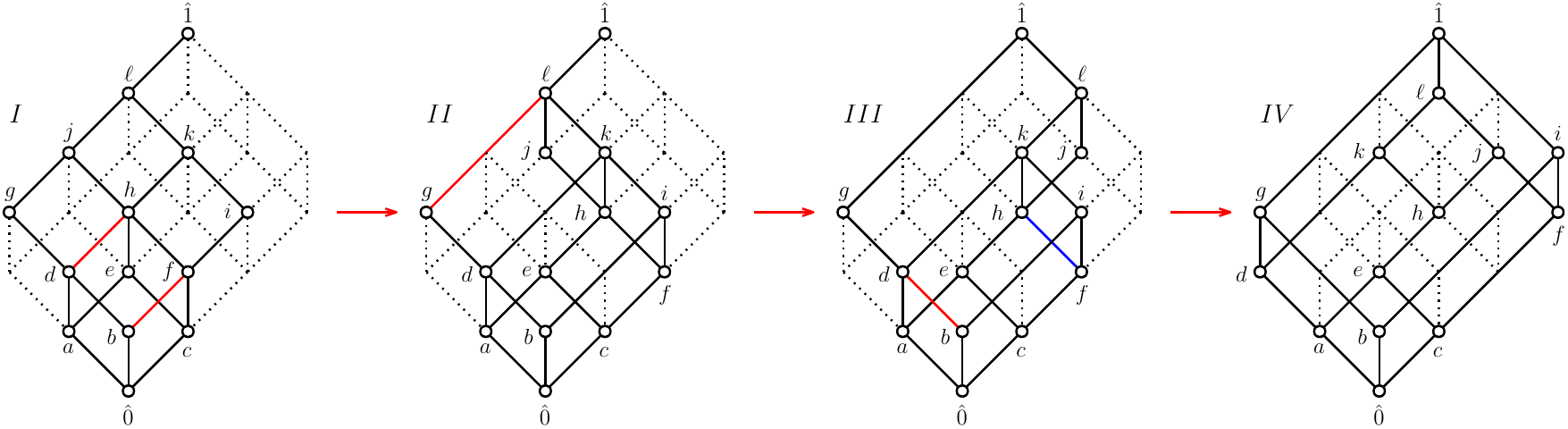}
        \caption{From the Dyck lattice to a Cambrian lattice (see Figure $7$ in~\cite{Rea1}).}
        \label{fig:camb}
    \end{figure}
\end{example}

Below is the list of $14$ linear intervals for each lattice; $12$ linear intervals are $3$-element chains and $2$ linear intervals are $4$-element chains. 

\begin{itemize}
    \item[$I$.] The $3$-element linear intervals are $[a, g]$, $[b, g]$, $[b, i]$, $[c, i]$, $[d, k]$, $[e, j]$, $[e, k]$, $[f, j]$, $[g, \ell]$, $[i, \ell]$, $[j, \hat{1}]$, $[k, \hat{1}]$, and the $4$-element linear intervals are $[g, \hat{1}]$ and $[i, \hat{1}]$.
    \item[$II$.] The $3$-element linear intervals are $[\hat{0}, f]$, $[a, g]$, $[a, h]$, $[b, g]$, $[c, i]$, $[e, j]$, $[e, k]$, $[f, j]$, $[g, \hat{1}]$, $[i, \ell]$, $[j, \hat{1}]$, $[k, \hat{1}]$, and the $4$-element linear intervals are $[a, j]$ and $[i, \hat{1}]$.
    \item[$III$.] The $3$-element linear intervals are $[\hat{0}, f]$, $[a, g]$, $[a, h]$, $[b, g]$, $[c, i]$, $[d, \ell]$, $[e, j]$, $[e, k]$, $[f, j]$, $[i, \ell]$, $[j, \hat{1}]$, $[k, \hat{1}]$, and the $4$-element linear intervals are $[a, j]$ and $[i, \hat{1}]$.
    \item[$IV$.] The $3$-element linear intervals are $[\hat{0}, d]$, $[\hat{0}, f]$, $[a, g]$, $[a, h]$, $[c, h]$, $[c, i]$, $[d, \ell]$, $[e, j]$, $[e, k]$, $[f, \ell]$, $[j, \hat{1}]$, $[k, \hat{1}]$, and the $4$-element linear intervals are $[a, j]$ and $[c, k]$.
\end{itemize}

The reasoning behind each refinement is as follows:

\begin{itemize}
    \item[$I \rightarrow II$.] Note the similarity with $I. \rightarrow II.$ from the previous example. The main difference here is the drawing of $\mathrm{Dyck}_4$, which in this case does not force the refinement of any interval outside of $k_\downarrow$. The maps between linear intervals are $[b, i]_I \rightarrow [\hat{0}, f]_{II}$, $[d, k]_I \rightarrow [a, h]_{II}$, $[g, \ell]_I \rightarrow [g, \hat{1}]_{II}$, and $[g, \hat{1}]_{I} \rightarrow [a, j]_{II}$. 
    \item[$II \rightarrow III$.] Note that $[g, \hat{1}]$ is the only linear interval with $[g, \hat{1}]$ in $II$ and there does not exist a linear interval with $[d, \ell]$ in $II$. This is the simplest case of refinement where we only have $[g, \hat{1}]_{II} \rightarrow [d, \ell]_{III}$.
    \item[$III \rightarrow IV$.] 
    \begin{itemize}
        \item[(i)] Geometrically, observe that $[d, \ell]$ lie along the same line segment. So, the refinement of $g_\downarrow$ result in $i$ being incomparable to $k$ and $\ell$ in $IV$. Note that we have $i \lessdot 1$ in $IV$ and $f \lessdot i$. So, informally speaking, the embedding into the grid forces $f \lessdot i$ to move up. As a result, we have the refinement of $[c, j]$.
        \item[(ii)] We have $[b, g]_{III} \rightarrow [\hat{0}, d]_{IV}$ in $g_\downarrow$ and $[f, j]_{III} \rightarrow [c, h]_{IV}$ in $[c, j]$. Otherwise, we have $[i, \ell]_{III} \rightarrow [j, \ell]_{IV}$ and $[i, \hat{1}]_{III} \rightarrow [c, k]_{IV}$. 
    \end{itemize}
\end{itemize}

\begin{remark}[On other lattices]
    There are other lattices that refine $\mathrm{Dyck}_4$ such that the number of linear intervals remains the same. For general $n$, we do not know how many of these lattices there are. Furthermore, we do not know of a procedure for obtaining Cambrian lattices using the Dyck lattice. \Cref{exm:refine_dyck_1} and \Cref{exm:refine_dyck_2} are here simply to show that these lattices exist.
\end{remark}

\subsection{Altitude lattices and their skeletal posets}

Recall from \Cref{thm:dtam} that the alt-Tamari lattices are anatomically related by \Cref{def:anatomic}. Given the lattices $\mathcal{L} \neq \mathrm{Dyck}_4$ in \Cref{exm:refine_dyck_1} and \Cref{exm:refine_dyck_2}, observe that $h(x_\downarrow)_\mathcal{L} = h(x_\downarrow)_{\mathrm{Dyck}_4}$. For visual clarity, consider the figures below.

\begin{figure}[htp]
    \centering
    \includegraphics[width=0.6\linewidth]{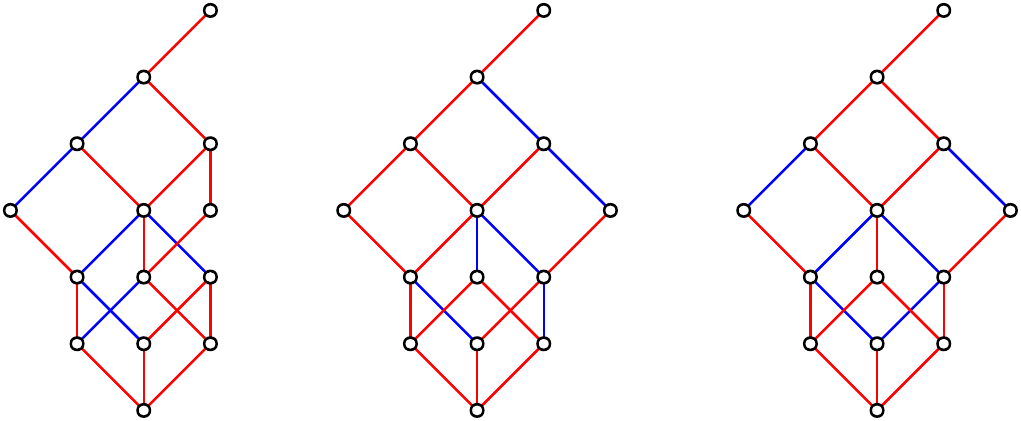}
    \caption{The Dyck lattice as a poset induced by embeddings of skeletal posets. \textcolor{red}{Cover relations} in \textcolor{red}{red} are those of the skeletal posets, and \textcolor{blue}{cover relations} in \textcolor{blue}{blue} are induced by the embeddings. The two leftmost embeddings come from the same skeletal poset.}
    \label{fig:dyck_embeddings}
\end{figure}

\begin{figure}[htp]
    \centering
    \includegraphics[width=0.8\linewidth]{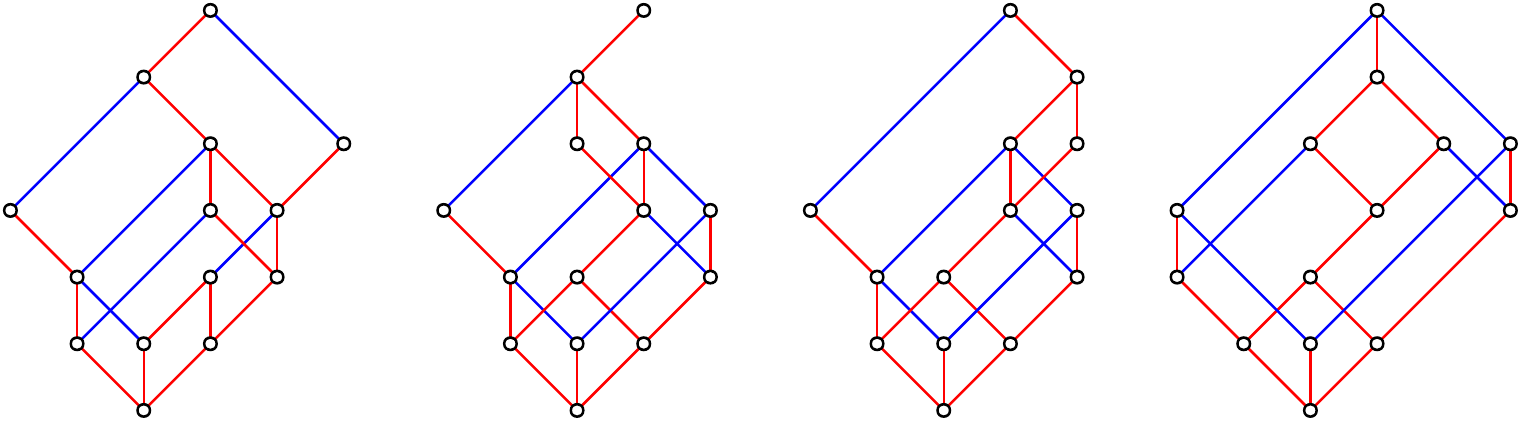}
    \caption{Embeddings of skeletal posets of the ``largest'' altitude lattices. \textcolor{red}{Cover relations} in \textcolor{red}{red} are those of the skeletal posets, and \textcolor{blue}{cover relations} in \textcolor{blue}{blue} are induced by the embeddings.}
    \label{fig:samples_sk}
\end{figure}

Every lattice in the above figures is realizable as a poset induced by embeddings of some skeletal posets, namely, the skeletal posets of the ``largest" lattices. Motivated by alt-Tamari lattices, we say that these lattices belong to the family of altitude lattices determined by the Dyck lattice. In turn, we are led to the following definition.

\begin{definition}[Altitude lattices]
\label{def:altitude_lattices}
    A family of \emph{altitude lattices} is a set of anatomically related lattices $\mathcal{L}$ with the property that any lattice in the set refine the same distributive lattice. 
\end{definition}

By \Cref{def:anatomic}, altitude lattices are also related via their skeletal posets. As seen in the previous figures, skeletal posets are, in the graph-theoretic sense, the largest spanning (directed) subgraph that many altitude lattices have in common. Within the context of extensions/refinements, skeletal posets are a useful tool for checking whether a particular altitude lattice extends/refines another. 

\begin{remark}[On \Cref{lem:same_lin} and skeletal posets]
    Note that the refinement of an interval isomorphic to $\mathrm{Dyck}_3$ in \Cref{exm:refine_dyck_1} and in \Cref{exm:refine_dyck_2} does not change the height $h(x_\downarrow)$ for any element $x$. In general, $h(x_\downarrow)$ is an invariant under refinements of Dyck intervals $\mathrm{Dyck}_3$ (resp., extensions of Tamari intervals $\mathrm{Tam}_3$). So, whenever we use \Cref{lem:same_lin}, the skeletal posets of altitude lattices act, informally, as records of cover relations that change or remain the same. Thus, if the reader is interested in using \Cref{lem:same_lin} for specific lattices, it may be of interest to study their skeletal posets.       
\end{remark} 

\subsection{Lattices and some extremal problems} 

While studying how embeddings of lattices in the grid interact with extensions/refinements that does not change the number of linear intervals in these lattices, we came across $2$ extremal problems of independent interest. 

One of these problems is already present in \Cref{exm:refine_dyck_1} and in \Cref{exm:refine_dyck_2}. Notably, it is a problem concerning the number of refinements starting from the Dyck lattice and, generally, a distributive lattice. Formally, the setup is as follows:

Let $\lambda = (\lambda_i)_{i \in \mathbb{N}}$ be a partition of $n$, i.e., $\sum_{i \in \mathbb{N}} \lambda_i = n$, where $n$ is the height of a distributive lattice $\mathcal{L}$. For any $\mathcal{L}$, there exist partitions $\lambda$ such that $\mathcal{L}$ can be embedded into a product of chains where the height of each chain is $\lambda_i$. For example, we have the partition $\lambda = (n-1, n-2, \dots, 1)$ for $\mathrm{Dyck}_n$. Without loss of generality, let $(\lambda_i)_{i \in \mathbb{N}}$ be a nonincreasing sequence. For distributive lattices $\mathcal{L}$ that are not chains or product of chains, fix a partition $\lambda$ such that $\mathcal{L}$ can be embedded into a product of chains where the height of each chain is $\lambda_i$. Consider the problem below.

\begin{problem}[The number of refinements of a distributive lattice]
    Find the minimal (resp., maximal) number of refinements starting from $\mathcal{L}$, as in \Cref{exm:refine_dyck_1} and \Cref{exm:refine_dyck_2}, such that the lattices at each refinement are embeddable into the product of chains recorded by $\lambda$, and the last refinement results in a ``largest'' lattice in that product. We use ``largest'' in the geometric sense and also to emphasize that the lattice need not be the coarsest, i.e., most refined lattice.         
\end{problem}

With the above problem in mind, recall that for $n=4$, there are $4$ alt-Tamari lattices. Notably, there are $2$ alt-Tamari lattices that are not $\mathrm{Dyck}_4$ or $\mathrm{Tam}_4$. Since neither of these $2$ lattices extends or refines the other, there exist two sequences of refinements from $\mathrm{Dyck}_4$ to $\mathrm{Tam}_4$ where every lattice is an alt-Tamari lattice. In fact, we suspect that the lattice $II.$ in \Cref{exm:refine_dyck_1} is an alt-Tamari lattice. Furthermore, it seems that the sequences coming from alt-Tamari lattices are minimal. We do not explore this direction in~\cite{L2} but highlight it here for the interested reader.    

Our second extremal problem focuses on embeddings of skeletal posets. Consider the figure below.

\begin{figure}[htp]
    \centering
    \includegraphics[width=0.65\linewidth]{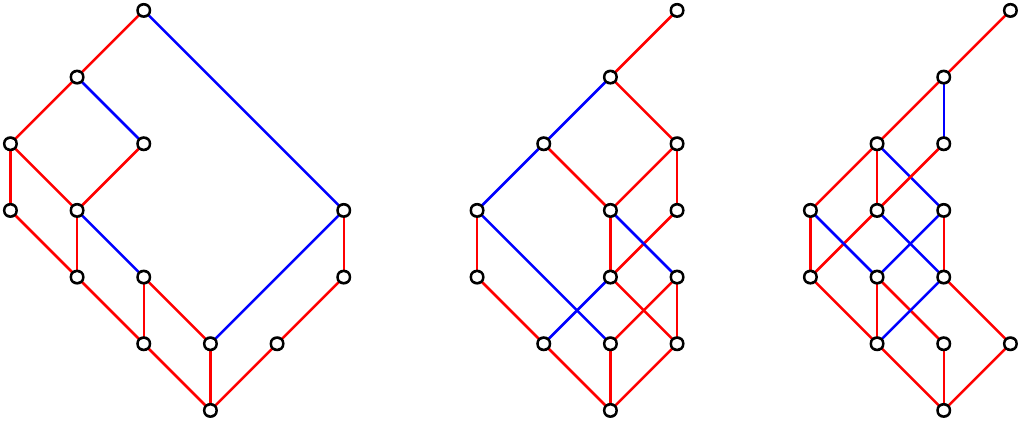}
    \caption{Embeddings of $\mathrm{SK}(\mathrm{Tam}_4)$ with $4, 5$ and $6$ induced \textcolor{blue}{cover relations} in \textcolor{blue}{blue} from left to right, respectively.}
    \label{fig:extremal}
\end{figure}

With the same setting as before, we consider a random embedding of a skeletal poset into the product of chains recorded by $\lambda$. Geometrically, recall that cover relations in this embedding correspond to line segments along some axes. For such an embedding, it is natural to count the minimal and maximal number of induced cover relations. As such, we suggest the following problem.

\begin{problem}[The number of induced cover relations]
    Consider a random embedding of a skeletal poset (of a ``largest'' lattice) in the product of chains recorded by $\lambda$. Let $c$ be the number of induced cover relations. Find the minimal (resp., maximal) value of $c$. In addition, does there exist a lattice for every value of $c$? And is there a $c$ such that the induced poset is always a lattice?   
\end{problem}

\section{Kneser graphs and reconstruction?}
\label{sec:kneser}

Unlike \Cref{sec:sk_tam} and \Cref{sec:alt}, this section is a very short note on Kneser graphs with very little original input. That is, we do not have new results on Kneser graphs. Instead, we give a definition of Kneser graphs for posets (with $\hat{0}$). Besides the obvious proposal to study the properties of these graphs, we suggest a direction of research with regards to poset reconstruction. For the sake of formality, let us recall the classic Kneser graphs and our definition of Kneser graphs. 

\begin{definition}[Kneser graphs]
\label{def:kneser}
    The \emph{Kneser graph} $KG_{n,k}$ has as vertices the $k$-element subsets of $[n]$, and two subsets are adjacent whenever they are pairwise disjoint. 
\end{definition}

\begin{definition}[Kneser graphs for posets with $\hat{0}$]
\label{def:kneser_poset}
    The \emph{Kneser graph} $KG(k)$ of a poset $\mathcal{P}$ with $\hat{0}$ has as vertices the elements $x$ where $h(x_\downarrow) = k$ for $1 \leq k \leq h(\mathcal{P})$ and two elements are adjacent whenever the intersection of their down sets is $\hat{0}$.
\end{definition}

\begin{remark}[On Kneser graphs for posets]
    Surprisingly, we could not find our definition of Kneser graphs in the literature. As such, to the best of our knowledge, these graphs have not been studied even for some well-known lattices, e.g., the Noncrossing partition lattice. 
\end{remark} 

Recall that the Boolean lattice is defined as the subsets of $[n]$ ordered by inclusion. Thus, our definition, i.e., \Cref{def:kneser_poset}, for the Boolean lattice recovers \Cref{def:kneser}. Moreover, it is straightforward to generalize \Cref{def:kneser_poset} by using more parameters, which may include other generalizations of Kneser graphs.

\begin{definition}[Generalized Kneser hypergraphs for posets]
\label{def:kneser_hypergraph_poset}
    The \emph{generalized Kneser hypergraph} $KG(r, s, k)$ of a poset $\mathcal{P}$ has as vertices the elements $x$ where $h(x_\downarrow) = k$ for $1 \leq k \leq h(\mathcal{P})$ and $r$ elements are adjacent whenever the height of the intersection of their down sets is at most $s \leq k - 1$.
\end{definition}

It is assumed in \Cref{def:kneser_hypergraph_poset} that $r$ is at most the maximal number of elements $x$ with $h(x_\downarrow) = k$. For $r=2, s=0$, we get \Cref{def:kneser_poset}. Later in the section, we note how the generalized Kneser hypergraphs of a poset $\mathcal{P}$ keep track of the ``disjointness'' of chains in $\mathcal{P}$. First, we give a brief history and motivation for Kneser graphs. 

Kneser graphs date back to Kneser's conjecture on their chromatic number. Explicitly, for the Kneser graph $KG_{n,k}$ with $n \geq 2k$, its chromatic number is $n - 2k + 2$. Kneser's conjecture was proved by Lovász using the Borsuk-Ulam theorem. It is a celebrated proof that led to topological methods in Combinatorics, or Topological Combinatorics (for further reading, we recommend the survey~\cite{DM}). Another noteworthy parameter of Kneser graphs is their independence number, defined as the maximum size of a set of pairwise nonadjacent vertices. Explicitly, it is the maximal number of $k$-element subsets such that their pairwise intersection is nonempty, which for $n \geq 2k$, is $\binom{n-1}{k-1}$. It is better known as the Erdös-Ko-Rado theorem (for further reading, we recommend the book~\cite{GM}). Finally, it was conjectured in the 1970s that Kneser graphs are Hamiltonian (except the Petersen graph or $KG_{5,2}$), i.e., they contain a cycle that visits every vertex exactly once. It is now a theorem due to Merino, Mütze, and Namrata~\cite{MMN}. In summary, the parameters of Kneser graphs encode some important results in Combinatorics (in particular, Extremal Combinatorics). 

With the above in mind, it is of interest to study ``Kneser graphs'' of different set systems. For our setting, the set systems here are represented by elements in some poset $\mathcal{P}$. The advantage of using $\mathcal{P}$ is that many set systems come equipped with a natural (partial) ordering. In the classical case, i.e., subsets of $[n]$ ordered by inclusion, this is the Boolean lattice as mentioned previously. The reader may be enticed to use and study \Cref{def:kneser_hypergraph_poset} for their posets. If so, we ask the reader to also consider another approach to Kneser graphs concerning poset reconstruction.

We mentioned earlier how the generalized Kneser hypergraphs for posets, or just the Kneser graphs for posets $\mathcal{P}$ with $\hat{0}$, keep track of ``disjointness'' of chains in $\mathcal{P}$. Notably, we have the following proposition, which follows easily from \Cref{def:kneser_poset}.

\begin{proposition}[Disjointness of chains]
    Consider the Kneser graph $KG(k)$ of a poset $\mathcal{P}$ with distinct vertices $x$ and $y$. Maximal chains $\mathcal{C}_x$ in $x_\downarrow$ and maximal chains $\mathcal{C}_y$ in $y_\downarrow$ are pairwise disjoint except at $\hat{0}$, i.e., $\mathcal{C}_x \cap \mathcal{C}_y = \hat{0}$, if and only if there is an edge between $x$ and $y$.
\end{proposition}

So, in the poset setting, Kneser graphs are ``disjointness'' graphs. Then, the generalized Kneser hypergraphs from \Cref{def:kneser_hypergraph_poset} are graphs that ``measure'' how disjoint the chains in some downsets are. Informally, when $\mathcal{P}$ is a lattice $\mathcal{L}$, these graphs, depending on the choice of $r, s,$ and $k$ in \Cref{def:kneser_hypergraph_poset}, tell us where the meet of some elements at a certain height is in $\mathcal{L}$. Since we define Kneser graphs for any $h(x_\downarrow) = k$ for $1 \leq k \leq h(\mathcal{P})$, the natural next question is, given all Kneser graphs of $\mathcal{P}$, that is, the Kneser graph for each value $k$, is it possible to determine or reconstruct some posets $\mathcal{P}$, especially when $\mathcal{P}$ is a lattice $\mathcal{L}$? 

The short and boring answer is yes. We say "boring" to emphasize that chains are trivially lattices determined by Kneser graphs with only one vertex. Furthermore, note that the relation $x_\downarrow \cap y_\downarrow = \hat{0}$ for $h(x_\downarrow) = h(y_\downarrow)$, and more generally $h(x_\downarrow \cap y_\downarrow) < h(x_\downarrow) = h(y_\downarrow)$, is finer than the incomparability condition. Simply put, vertices $x$ and $y$ in Kneser graphs are always incomparable elements in $\mathcal{P}$. So, the question above of determining or reconstructing $\mathcal{P}$ is really about using only specific subgraphs of the incomparability graph to determine or reconstruct $\mathcal{P}$. Since the incomparability graph of $\mathcal{P}$ need not be unique, i.e., two non-isomorphic posets $\mathcal{P}$ may have the same incomparability graph, the scope of interesting posets for the question we posed is even more limited.

For example, if we take all elements instead of just elements at a certain height with the relation $x_\downarrow \cap y_\downarrow$, we recover the zero-divisor graph (for further reading, we recommend the survey~\cite{AB}). For determining posets, the zero-divisor graph is very ineffective as it can not differentiate between chains and lattices where only one element covers $\hat{0}$. As ideas go, this is when we have a light bulb moment: what if we use Kneser graphs in combination with other subgraphs of the incomparability graph? 

Notably, distinct posets $\mathcal{P}$ and $\mathcal{Q}$ may have the same zero-divisor graph, but not the same Kneser graphs, and vice versa. Indeed, the zero-divisor graph records ``disjointess'' across all elements, while Kneser graphs are limited to elements at the same height. In other words, Kneser graphs act as another incomparability invariant. Perhaps unsurprisingly to the reader given \Cref{fig:sk_tam_3} and \Cref{sec:tam}, a good example of posets $\mathcal{P}$ and $\mathcal{Q}$ with the same Kneser graphs but different zero-divisor graphs, and more generally incomparability graphs, is $\mathcal{P} = \mathrm{Dyck}_3$ and $\mathcal{Q} = \mathrm{Tam}_3$. We end this note with the following problem.

\begin{problem}[Posets determined by a pair of incomparability invariants]
    Find pairs of incomparability graphs, or subgraphs of the incomparability graph with all elements of a poset $\mathcal{P}$ except its minimal and maximal elements, and (generalized) Kneser (hyper)graphs that uniquely determine a poset $\mathcal{P}$, or more generally, a class of posets $\mathcal{P}$.  
\end{problem}

\section*{Acknowledgement}

The author is grateful for many discussions with Clément Chenevière, Titouan Galor, and Viviane Pons with regard to \Cref{sec:sk_tam}. The author thanks, in advance, those who aid him in the contents and the writing of his full works on \Cref{sec:sk_tam} and \Cref{sec:alt}. They will be given their due as the list of people grows. The author apologizes if he misses anyone for this preliminary version.

\end{document}